\documentclass[11pt,a4paper]{amsart}
\usepackage{geometry, empheq}
\usepackage[english]{babel}
\usepackage{amsmath}
\usepackage{amssymb}
\usepackage{amsthm}
\usepackage{graphicx}
\usepackage{color,xcolor}
\usepackage{varioref}
\usepackage{nicefrac}
\usepackage{subcaption}
\usepackage{float}
\usepackage[normalem]{ulem}
\usepackage{multirow}
\usepackage{multicol}
\usepackage[scaled]{beramono}
\usepackage[framed,numbered,autolinebreaks,useliterate]{mcode}
\usepackage{upquote}
\usepackage[margin=1cm,font=scriptsize,labelfont=bf]{caption}
\usepackage{color,xcolor}
\usepackage{listings}
\usepackage{algorithmicx}
\usepackage{algorithm}
\usepackage{algpseudocode}
\usepackage{algpascal}
\usepackage{enumitem}
\usepackage{mathtools}
\usepackage{booktabs}
\usepackage{tikz}
\usepackage{pgfplots}
\usepackage{adjustbox}
\usepackage{lipsum}

\newtheorem{lemma}{{Lemma}}[section]
\newtheorem{theorem}{{Theorem}}[section]

\newtheorem{remark}{{Remark}}[section]

\newtheorem{assumption}{{Assumption}}[section]

\title[A bilevel learning approach for optimal observation placement]{A bilevel learning approach for optimal observation placement in variational data assimilation}
\author{P. Castro$^\dag$, J.C. De los Reyes$^\dag$}
\address{$\dag$ Research Center on Mathematical Modeling (MODEMAT), Escuela Polit\'ecnica Nacional, Quito, Ecuador.}
\thanks{Paula Castro acknowledges partial support from the Master Program in Mathematical Optimization at Escuela Politécnica Nacional de Ecuador}

\begin{document}
\begin{abstract}
In this paper we propose a bilevel optimization approach for the placement of space and time observations in variational data assimilation problems. Within the framework of supervised learning, we consider a bilevel problem where the lower-level task is the variational reconstruction of the initial condition of a semilinear system, and the upper-level problem solves the optimal placement with help of a sparsity inducing function. Due to the pointwise nature of the observations, an optimality system with regular Borel measures is obtained as necessary optimality condition for the lower-level problem. The latter is then considered as constraint for the upper-level instance, yielding an optimization problem constrained by a multi-state system with measures. We demonstrate existence of Lagrange multipliers and derive a necessary optimality system characterizing the optimal solutions of the bilevel problem. The numerical solution is carried out also on two levels. The lower-level problem is solved using a standard BFGS method, while the upper-level one is solved by means of a projected BFGS algorithm, based on the estimation of $\epsilon$-active sets. Finally some numerical experiments are presented to illustrate the main features of our approach.
\end{abstract}
\maketitle

\section{Introduction}
Data assimilation (DA) problems deal with the reconstruction of the initial condition of a dynamical system based on observations of the state, previous estimates and the system model. They are widely used in applications like, e.g., numerical weather prediction. For these problems, the more accurate the initial condition, the better the quality of the forecast of the system state. There are several approaches to deal with data assimilation problems, such as, optimal interpolation, variational approaches or hybrid methods (see \cite{Kal} and the references therein).

Variational approaches focus on solving an optimal least squares problem, and can be broadly classified in two classes, depending on the type of observations considered. The first one, three-dimensional variational analysis ($3D-$VAR), considers observations in just one instant of time, while the second one, four-dimensional variational analysis ($4D-$VAR), takes into account observations distributed in a given period of time $[t_0, t_n]$.

Depending on the amount of observations, data assimilation methods may be more or less efficient in reconstructing a useful initial condition. In meteorology, for instance, data may be collected through meteorological stations, satellite images, radiosondes, among others. In practice, the installation and operational costs of such observation devices may be too high and it is important to locate them in an optimal way, meaning that as few as possible should be placed and the richest amount of information should be measured.

Several approaches have emerged in the last decades in order to optimally place sensors in different settings. The most classical one is optimal filtering, proposed in the 70's to cope with Gaussian linear problems \cite{balakrishnan1967state,bensoussan1971,Curtain1978}, and further developed in \cite{burns1994optimal,burnsraut2015,HintRaut} to deal with nonlinear systems. An alternative observability approach for the location of sensors in linear parabolic and hyperbolic equations was recently developed in \cite{privat2015optimal}.
Additionally, in \cite{OED,Al} an A-optimal experimental design approach for the location of sensors in systems governed by PDE's was developed, which was also applied to a thermo-mechanical data assimilation problem in \cite{herzogriedel2015}.

In this paper we tackle the optimal placement problem using a bilevel learning approach \cite{Haber2010,dlReSchoen2013,hollerkunisch2018}. In contrast to optimal experimental design strategies, our framework allows us to work with different quality measures and is not restricted to the A-, D- or E- optimal experimental design paradigms \cite{pukelsheim}. Moreover, differently from previous related contributions \cite{Haber2010,OED,Al}, we are able to analyze the resulting bilevel optimization problem in function spaces and get an insight into its complex mathematical structure. By considering measures in space and mollified Dirac measures in time, we treat the lower-level optimality system as a multi-state system of time-dependent PDE with measures. Due to the possible multiplicity of solutions, we consider an adapted penalty approach in order to proof existence of Lagrange multipliers and derive an optimality system for the bilevel learning problem solutions.

The numerical solution is carried out in two stages. For the upper-level problem we use a projected BFGS method, whose inverse Hessian approximation is iteratively built upon the estimation of $\epsilon-$active sets, while, for the lower-level problem, a standard BFGS algorithm is considered. To further enhance the sparsity of the solution vector, the linear penalization function is replaced by a concave one, with values between $0$ and $1$.

The proposed bilevel optimization framework for observation placement, both in space an time, in the context of variational data assimilation, as well as the rigorous mathematical analysis of the lower- and upper-level problems constitute the genuine contribution of this manuscript. This is further complemented by the design of a second-order numerical algorithm for the solution of the problem, whose performance is computationally verified.

The structure of the paper is as follows. In Section 2 we study the variational data assimilation problem, and discuss existence und uniqueness of the solution to the adjoint
equation with regular Borel measures on its right-hand side. Section 3 focuses on the existence of Lagrange multipliers and on the rigorous derivation of an optimality system for the bilevel optimal placement problem. In Section 4, we present a second-order solution algorithm for the problem and discuss convergence properties as well as numerical aspects. Finally, in the last section, several computational experiments are carried out to verify the main properties of the approach and the solution method.

\section{Variational Data Assimilation Problem}
In this paper, we focus on the $4D-$VAR approach to solve semilinear data assimilation problems. Specifically, we consider the problem:
\begin{multline}\label{eq:201}
  \displaystyle
  \min_{u}J(y,u)=\dfrac{1}{2}\displaystyle\int\limits_0^T\sum_{k,i} w_k\sigma_i \rho_i(t)[y(x_k,t)-z_o(x_k,t)]^2~ dt\\+ \frac{1}{2} \|u-u_b\|_{B^{-1}}^2+\frac{\vartheta}{2}\|\nabla (u-u_b)\|^2_{L^2(\Omega)}
\end{multline}
\begin{equation}\label{eq:202}
   \begin{array}{lrll}
  &\frac{\partial y}{\partial t}+ Ay +g(y)= &0 &\text{ in }Q=\Omega\times]0,T[\\
\text{subject to: }&y=&0&\text{ on }\Sigma=\Gamma\times]0,T[\\
  &y(x,0)=&u&\text{ in }\Omega,
  \end{array}
\end{equation}
where $z_o(x_k,\cdot)$ represent the state observations at the spatial point $x_k$, $u_b \in H^1(\Omega)$ and $B^{-1} \in\mathcal{L}(L^2(\Omega))$ are the background information and the background information error covariance operator, respectively. As usual for data assimilation problems, this covariance operator is built using a Bayesian approach and assuming Gaussian noise in the data. The general form of a $4D-$VAR problem includes also observation error covariance matrices, which, for the sake of readability, we set equal to the identity. We also introduce a location vector $w=(w_k)$, $k=1,\ldots,n_s$, whose components take the value one if the placement $x_k$ is chosen and zero otherwise. Moreover, to be able to choose not only the optimal locations but also optimal time subintervals, we consider a  vector $\sigma=(\sigma_i)$, $i=1,\ldots,n_T$, where each component takes the value one if the time subinterval $i$ is chosen and zero if not.

In \eqref{eq:202} $A$ stands for a linear elliptic second order differential operator of the form
\[
Ay(x)=- \displaystyle\sum_{i,j}^n D_i(a_{ij}(x)D_jy(x)),\quad \text{for }x\in\Omega,
\]
with regular coefficients $(a_{ij}(x))$, for $i,j=1,\dots,n$, satisfying the symmetry condition $a_{ij}(x)=a_{ji}(x)$ and the condition of uniform ellipticity. Additionally, the nonlinear term in equation \eqref{eq:202} is assumed to verify the following conditions:

\begin{assumption}\label{assum:001}
  \hfill
  \begin{enumerate}
    \item $g=g(x,t,y):Q\times \mathbb{R}\mapsto\mathbb{R}$ satisfies the Carath\'eodory conditions and is uniformly bounded at the origin, i.e., $|g(x,t,0)|\leq K$, for some $K>0$,
    \item $g$ is monotone increasing with respect to $y$ for almost every $(x,t)\in Q$,
    \item $g$ is twice continuously differentiable with respect to $y$ and
      \begin{equation*}
      |g_y(x,t,y)|+|g_{yy}(x,t,y)| \leq K,
    \end{equation*}
      for some $K>0$, for almost all $(x,t)\in Q$  and any $y \in \mathbb{R}$.
  \end{enumerate}
\end{assumption}

Furthermore, $\|u-u_b\|_{B^{-1}}^2:=\int\limits_{\Omega}(u-u_b)B^{-1}(u-u_b) ~dx$ and $\vartheta >0$ is a regularization parameter, which in practice can be as small as required. In \eqref{eq:201}, we also consider regular support functions $\rho_i(t)\in C^2(0,T)$, for each $i=1,\ldots,n_T$, that act as mollifiers in short time-interval observations (almost instantaneous).

\subsection{Existence of a solution for the DA problem}
Let $\Omega\subset\mathbb{R}^m$, $1< m\leq 3$, be a bounded domain of class $\mathcal C^2$, $Q:=\Omega\times(0,T)$ and $\Sigma:=\Gamma\times(0,T)$, with $T>0$ a fixed real number. The constraint of the data assimilation problem, given by \eqref{eq:202}, is well-posed in the required high-regularity spaces, as will be stated in the following theorem.
\begin{theorem}\label{teo:existence}
  If $u\in H_0^1(\Omega)$ and the nonlinear term verifies Assumption \ref{assum:001}, then equation \eqref{eq:202} has a unique solution $y\in H^{2,1}(Q)$.
\end{theorem}
\begin{proof}
  From \cite[Chapter 2]{lions1969quelques} we know that under Assumption \ref{assum:001} and taking $u\in L^2(\Omega)$, the semilinear equation has a unique weak solution $y \in W(0,T)$. The desired regularity follows by applying a standard boot-strapping argument to the auxiliar linear equation:
\begin{equation}\label{eq:bootstap}
  \begin{aligned}
    \frac{\partial \phi}{\partial t}+ A\phi &=-g(y)&\text{ in }Q,\\
    \phi&=0&\text{ in }\Sigma,\\
    \phi(x,0)&=u&\text{ in }\Omega.
  \end{aligned}
\end{equation}
Thanks Assumption \ref{assum:001}, it holds that $g(y)\in L^2(Q)$. Moreover, using the higher regularity of the initial condition, $u\in H_0^{1}(\Omega)$, it follows that $\phi\in H^{2,1}(Q)$ is the unique solution to equation \eqref{eq:bootstap} \cite[Theorem 9.1]{ladyzhenskaia1968linear}. Additionally, it verifies that
\begin{equation}\label{eq:wf}
  \displaystyle
  -\int\limits^T_{0}\int\limits_{\Omega}\phi v_t dx dt + \int\limits^T_{0}\int\limits_{\Omega}a\nabla \phi.\nabla v +g(y) v dxdt=\int\limits_{\Omega}u v(0) dx - \int\limits_{\Omega} \phi(T)v(T)dx,
\end{equation}
for all $v\in W(0,T)$. Therefore, $\phi\in H^{2,1}(Q)$ solves equation \eqref{eq:202} as well. The result holds due to the uniqueness of the solution to \eqref{eq:202}.
\end{proof}

\begin{remark}\label{rem:001}
Using the estimate given in \cite[Theorem 9.1]{ladyzhenskaia1968linear} and taking into account that, thanks to Assumption \ref{assum:001}, $g(y)\in L^2(Q)$, the following relation holds:
\begin{equation}\label{eq3:003}
  \| y\|_{H^{2,1}(Q)}\leq c(1+\| u\|_{H_0^1(\Omega)}), \text{ for some constant }c>0.
\end{equation}
\end{remark}

We introduce the control-to-state mapping $S:H_0^1(\Omega)\rightarrow H^{2,1}(Q)$, $u\mapsto S(u)=y$, as the solution operator associated with the semilinear parabolic equation \eqref{eq:202}. In the next two results we verify continuity and differentiability properties of this mapping.

\begin{theorem}\label{teo:S_continuity}
The control-to-state mapping is sequentially weakly continuous from $H_0^1(\Omega)$ to $H^{2,1}(Q)$.
\end{theorem}
\begin{proof}
Let $u_n\rightharpoonup u$ in $H_0^1(\Omega)$ and set $y_n:=S(u_n)$ and $y=S(u)$. Thanks to estimate \eqref{eq3:003}, $\{y_n\}$ is bounded in ${H^{2,1}(Q)}$. Consequently, due to the reflexivity of $H^{2,1}(Q)$, there exists a weakly convergent subsequence, denoted the same, and a limit point  $y\in H^{2,1}(Q)$ such that  $y_n\rightharpoonup y$ in $H^{2,1}(Q)$ as $n\to\infty$. Using the compact embedding $H^{2,1}(Q) \hookrightarrow \hookrightarrow L^\mu(Q), ~\mu \leq 10$ (see, e.g., \cite[p.13]{lions1983controle}), we get that  $y_n\to y$ in $L^\mu(Q)$ as $n\to\infty$. Thanks to the continuity of $g$, $g(y_n)\to g(y)$ strongly in $L^2(Q)$. Moreover, due to the linearity and continuity of $\partial_t + A : H^{2,1}(Q) \to L^2(Q)$, we obtain that
\begin{equation*}
  \partial_t y + A y + g(y) = \lim_{n \to \infty} \partial_t y_n + A y_n + g(y_n)=0,
\end{equation*}
and the result follows.
\end{proof}

\begin{theorem}\label{teo:S_diff}
  The control-to-state mapping $S$ is G\^ateaux differentiable. Its derivative, in direction $h \in H_0^1(\Omega)$, is given by $S'(u)h=\eta$, where $\eta \in H^{2,1}(Q)$ corresponds to the unique solution of:
  \begin{equation}\label{eq:S_diff}
   \left\{
     \begin{array}{rll}
    \frac{\partial \eta}{\partial t}+ A\eta + g'(y)\eta= &0 &\text{ in }Q\\
    \eta=&0&\text{ on }\Sigma\\
    \eta(x,0)=&h&\text{ in }\Omega,
    \end{array}
    \right.
  \end{equation}
\end{theorem}

\begin{proof}
Let us consider the increment $y_\tau=S(u+\tau h)$, with $\tau>0$ and $h\in H_0^1(\Omega)$, and   $y=S(u)$. Substracting the corresponding equations, we get
  \begin{equation}\label{eq:diff}
    \begin{aligned}
      \frac{\partial}{\partial t} \left(\frac{y_\tau - y}{\tau}\right)+ A\left(\frac{y_\tau - y}{\tau}\right) + \frac{g(y_\tau)-g(y)}{\tau}&=0&\text{in }Q,\\
      \frac{y_\tau - y}{\tau}&=0 &\text{in }\Sigma\\
      \frac{y_\tau - y}{\tau}(x,0)&=h &\text{in }\Omega.
    \end{aligned}
  \end{equation}
Due to the monotonicity and Lipschitz continuity of $g$, $\{\frac{y_\tau - y}{\tau}\} $ solution of equation \eqref{eq:diff} is bounded in $W(0,T)$. Applying a bootstrapping argument further implies that $\{\frac{y_\tau - y}{\tau}\}$ is bounded in $H^{2,1}(Q)$. Therefore, there exists a subsequence, denoted the same, that is weakly convergent to some $\hat{\eta}\in H^{2,1}(Q)$. Equation \eqref{eq:S_diff} is well-posed in $H^{2,1}(Q)$, since, thanks to Assumption \ref{assum:001}, $g'(y)$ is bounded by a uniform constant. Substracting equations \eqref{eq:diff} and  \eqref{eq:S_diff}, and denoting $\hat{y}:=\frac{y_\tau - y}{\tau}-\eta$, we get
\begin{equation}\label{eq:diff_1}
  \begin{aligned}
    \frac{\partial \hat{y}}{\partial t}+ A\hat{y} + \frac{1}{\tau}\left(g(y_\tau)-g(y)-\tau g'(y)\eta\right)&=0&\text{in }Q,\\ \hat{y} &=0 &\text{in }\Sigma,\\ \hat{y}(x,0) &=0 &\text{in }\Omega,
  \end{aligned}
\end{equation}
from where, thanks to the G\^ateaux differentiability of the nonlinear term $g$ from $L^2(Q)$ to $L^2(Q)$, it follows that $\|\frac{y_\tau - y}{\tau}-\eta\|_{H^{2,1}(Q)}\to 0$ as $\tau\to0$. Consequently, from the uniqueness of the limit it holds that $\hat{\eta}=\eta$.
\end{proof}

Note that thanks to the compact embedding $H^{2,1}(Q)\hookrightarrow L^2(0,T;C(\bar{\Omega}))$ \cite[p.31]{casas2013parabolic}, the continuity of the state in the spatial variable holds, and the pointwise evaluations of $y$ in the spatial component are well--defined. Next we prove existence of an optimal solution to the data assimilation problem.

\begin{theorem}\label{teo:201}
Suppose that Assumption \ref{assum:001} holds. The data assimilation problem \eqref{eq:201} has at least an optimal solution $\bar{u}\in H_0^1(\Omega)$ with $\bar{y}=S(\bar{u})\in H^{2,1}(Q)$ its related optimal state.
\end{theorem}

\begin{proof}
Let $\{ (y_n,u_n) \}_{n\geq 1} \subset H^{2,1}(Q) \times H_0^1(\Omega)$ be a minimizing sequence, with $y_n=S(u_n)$. Since the cost functional $J$ satisfies
\begin{equation}\label{eq:101_1}
J(y(u),u)\geq \frac{1}{2} \displaystyle
 \|u-u_b\|_{B^{-1}}^2+\frac{\vartheta}{2}\|\nabla (u-u_b)\|_{L^2(\Omega)}^2\geq c_B\|u-u_b\|^2_{H^1(\Omega)},\quad c_B>0,
 \end{equation}
the sequence $\{ u_n \}$ is bounded in $H_0^1(\Omega)$ and there exists a subsequence, denoted the same, and a limit $\bar u\in H_0^1(\Omega)$ such that $u_n \rightharpoonup \bar u$ in $H_0^1(\Omega)$ as $n\to\infty$. Moreover, thanks to estimate \eqref{eq3:003}, $\{y_n\}$ is also bounded in $H^{2,1}(Q)$ and, up to a subsequence, $y_n \rightharpoonup \bar y$ weakly in $H^{2,1}(Q)$. Proceeding as in the proof of Theorem \ref{teo:S_continuity} we obtain that $g(y_n) \to g(\bar y)$ in $L^2(Q)$. By passing to the limit in equation \eqref{eq:202} we obtain that $\bar{y}$ is the solution to the semilinear equation corresponding to the initial condition $\bar{u}$.
Since $J$ is weakly lower semicontinuous,
  \[
  \displaystyle \inf_{u\in H_0^1(\Omega)} J(y(u),u)= \lim\inf_{n\to\infty} J(u_n) \geq J(\bar{y},\bar{u})
  \]
  and the result follows.
\end{proof}
Considering the control-to-state mapping, we can rewrite the cost functional of the data assimilation problem in reduced form as
\begin{multline}\label{eq:204}
  \min_u f(u) = \dfrac{1}{2}\displaystyle \int\limits_0^T\int\limits_\Omega\sum_{k,i}w_k\sigma_i\rho_i(t)[S(u)(x,t)-z_o(x,t)]^2\delta(x-x_k) dxdt \\ + \frac{1}{2} \|u-u_b\|_{B^{-1}}^2+\frac{\vartheta}{2}\|\nabla (u-u_b)\|_{L^2(\Omega)}^2
    \end{multline}

\subsection{Adjoint equation}
As a preparatory step for the derivation of an optimality system for the variational DA problem, we study the following adjoint equation:
\begin{equation}\label{eq:aux}
\begin{array}{rll}
\displaystyle
-\frac{\partial p}{\partial t}+A^*p +g'(y) p =&  \displaystyle\mu &\text{ in }Q\\
p=&0&\text{ on }\Sigma\\
p(x,T)=&0&\text{ in }\Omega,
\end{array}
\end{equation}
where $\mu \in L^2(0,T;\mathcal M(\Omega)).$ Equation \eqref{eq:aux} has to be understood in a very weak form, i.e., for all $\zeta\in H^{2,1}(Q)$,
\[
\displaystyle \iint\limits_Q\left(\frac{\partial\zeta}{\partial t}+A\zeta+g'(y)\zeta\right)p~dxdt+
\int\limits_{\Omega}p(0) \zeta(0)~dx=\int\limits_0^T\langle \mu(t),\zeta(t) \rangle_{\mathcal{M}(\Omega),C(\bar{\Omega})} dt.
\]

\begin{theorem}\label{lemm:adjoint}
  If $\mu\in L^2(0,T;W^{-1,r}(\Omega))$, then problem \eqref{eq:aux} has a unique solution $p\in L^2(0,T;W_0^{1,r}(\Omega))$, with $r\in[1,\frac{m}{m-1}[$. Moreover, there exists a constant $c_p>0$ such that
  \begin{equation}\label{casas:2}
    \| p \|_{L^2(0,T;W_0^{1,r}(\Omega))}\leq c_p \| \mu\|_{L^2(0,T;W^{-1,r}(\Omega))}.
  \end{equation}
\end{theorem}
\begin{proof}
Let us set $r'$ as the conjugate exponent of $r$, i.e., $\frac{1}{r} + \frac{1}{r'} =1$. Considering the maximal parabolic $L^2(0,T;W^{-1,r'}(\Omega))-$regularity of the operator $A$ \cite[p.2224]{meyer2017optimal}, we get that the operator $A+g'(y)$ also verifies the maximal $L^2(0,T;W^{-1,r'}(\Omega))-$regularity property \cite[p.2216]{meyer2017optimal}.
Consequently, by \cite[Lemma 36]{herzog2017existence}, the adjoint operator $(\partial_t + A + g'(y))^*$, is an isomorphism, such that, for all $\mu\in L^2(0,T;W^{-1,r}(\Omega))$, there exists a unique $p\in L^2(0,T;W_0^{1,r}(\Omega))$, with $\frac{\partial p}{\partial t}\in L^2(0,T;W^{-1,r}(\Omega))$,
solution of equation \eqref{eq:aux}, for every $1\leq r<\frac{m}{m-1}$. Estimate \eqref{casas:2} follows immediately using the maximal regularity of the operator (see e.g. \cite[p.8]{disser2017holder}).
\end{proof}

Hereafter we will use the simplified notation $L^2(W_0^{1,r}(\Omega)):=L^2(0,T;W_0^{1,r}(\Omega)).$
\begin{remark}\label{rem:002}
  \hfill
  \begin{itemize}
    \item Since $L^2(0,T;\mathcal{M}(\Omega))\hookrightarrow L^2(0,T;W^{-1,r}(\Omega))$, the result of Theorem \ref{lemm:adjoint} is also verified if we take the right-hand side of equation \eqref{eq:aux} in the space of weakly measurable functions $L^2(0,T;\mathcal{M}(\Omega))$.
    \item Considering the right-hand side of \eqref{eq:aux} in $\mathcal{M}(Q)$ leads to an ill-posed problem \cite[p.29]{casas2013parabolic}.
    \item The continuity of the solution of $\eqref{eq:aux}$ in the time variable is also verified. Specifically, $p\in C(0,T;L^2(\Omega))$ \cite[Remark 2.3]{casas2013parabolic}.
    \item Since $p\in L^2(W_0^{1,r}(\Omega))$ with $\frac{\partial p}{\partial t}\in L^2(0,T;W^{-1,r}(\Omega))$, the adjoint state $p$ also belongs to the maximal regularity space $\mathbb{W}_0^2(W^{1,r}_0,W^{-1,r}):=L^2(W_0^{1,r}(\Omega))\cap W_0^{1,2}(0,T;W^{-1,r}(\Omega))$.
    From \cite[p.8]{disser2017holder} the following estimate is verified:
    \begin{equation}\label{dissier_p}
      \| p \|_{\mathbb{W}_0^2(W^{1,r}_0,W^{-1,r})}\leq c_p \| \mu\|_{L^2(0,T;W^{-1,r}(\Omega))},\quad\text{ for some }c_p>0.
    \end{equation}
  \end{itemize}
\end{remark}

\subsection{Optimality system}
In this section we state the first order optimality conditions satisfied by local solutions $(\bar{u},\bar{y})\in H_0^1(\Omega) \times H^{2,1}(Q)$ of the variational data assimilation problem \eqref{eq:204}.

\begin{theorem}\label{teo:202}
Let $\bar{u}\in H_0^{1}(\Omega)$ be a solution to \eqref{eq:204} with $\bar{y}=S(\bar{u})\in H^{2,1}(Q)$ its associated optimal state. Then, there exists a unique adjoint state $\bar{p}\in L^2(W_0^{1,r}(\Omega))$, with $r\in\left[1,\frac{m}{m-1}\right[$, satisfying:\\

\begin{subequations}
      \noindent \underline{State equation }(in strong form):
      \begin{equation}\label{eq:202_2}
      \begin{array}{rll}
      \frac{\partial \bar{y}}{\partial t}+A\bar{y} +g(\bar{y})= &0 &\text{ in }Q,\\
      \bar{y}=&0&\text{ on }\Sigma,\\
      \bar{y}(x,0)=&\bar{u}&\text{ in }\Omega.
      \end{array}
      \end{equation}
      \underline{Adjoint equation }(in very weak form):
      \begin{equation}\label{eq:206}
      \begin{array}{rll}
      \displaystyle
      -\frac{\partial \bar{p}}{\partial t}+A^*\bar{p} + g'(\bar{y}) \bar{p}=&  \displaystyle\sum_{k,i} w_k\sigma_i\rho_i(t)\left[  \bar{y}(x,t)-z_o(x,t)\right]\otimes\delta(x-x_k) &\text{ in }Q,\\
      \bar{p}=&0&\text{ on }\Sigma,\\
      \bar{p}(x,T)=&0&\text{ in }\Omega.
      \end{array}
    \end{equation}
      \underline{Gradient equation }(in weak form):
      \begin{equation}\label{eq:207}
      \begin{array}{rll}
      \displaystyle
      -\vartheta\Delta (\bar{u}-u_b) +B^{-1}(\bar{u}-u_b)+\bar{p}(0)=& 0 &\text{ in }\Omega,\\
      \bar{u}=&0&\text{ on }\Gamma.
      \end{array}
      \end{equation}
\end{subequations}
\end{theorem}
\begin{proof}
  Since the solution operator is G\^ateaux differentiable, by applying the chain rule to the reduced cost functional, we get, for a given direction $h\in H_0^{1}(\Omega)$,
  \begin{multline*}
  f'(\bar{u})h=\displaystyle\int\limits_{\Omega}\int\limits_0^T\sum_{k,i} w_k\sigma_i\rho_i(t)[\bar{y}(x_k,t)-z_o(x_k,t)]\delta(x-x_k)S'(\bar{u})h ~dxdt\\
  +\int\limits_{\Omega}(\bar{u}(x)-u_b(x))B^{-1}h ~dx +\vartheta\int\limits_{\Omega} \nabla(\bar{u}-u_b).\nabla h ~dx,
  \end{multline*}
where $S'(\bar{u})h=\eta\in H^{2,1}(Q)$ is the solution to the linearized equation \eqref{eq:S_diff}. Moreover, using the adjoint equation \eqref{eq:206} in very weak form and Green's formula,
  \begin{multline*}
  f'(\bar{u})h=\displaystyle\int\limits_{\Omega}\int\limits_0^T \bar{p} \left(\dfrac{\partial \eta} {\partial t}+ A \eta + g'(\bar{y})\eta \right) dxdt + \int\limits_{\Omega}\bar{p}(0) \eta(0) ~dx
  \\+\int\limits_{\Omega}\left(\bar{u}(x)-u_b(x)\right)B^{-1}h ~dx +\vartheta\int\limits_{\Omega}\nabla (\bar{u}-u_b).\nabla h ~dx=0, \quad \forall h\in H_0^1(\Omega).
\end{multline*}
Considering the linearized problem \eqref{eq:S_diff} in the equation above, we then get
  \begin{equation}\label{eq:401_1}
  \displaystyle\int\limits_{\Omega}(\bar{u}(x)-u_b(x))B^{-1}h ~dx +\vartheta\int\limits_{\Omega} \nabla(\bar{u}-u_b).\nabla h ~dx +\int\limits_{\Omega}\bar{p}(0)h~dx=0,\forall h\in H_0^1(\Omega),
\end{equation}
which corresponds to the weak form of \eqref{eq:207}. Since the right hand side in \eqref{eq:206} belongs to $L^2(0,T;\mathcal M(\Omega)),$ it follows (see Remark \ref{rem:002}) that $\bar{p}\in L^2(W_0^{1,r}(\Omega))$.
\end{proof}

\begin{remark}
Using estimates \eqref{casas:2} and  \eqref{dissier_p} into the adjoint equation \eqref{eq:206} and applying Young's inequality, it can be verified that the adjoint state $\bar{p}\in L^2(W_0^{1,r}(\Omega))$, $r\in[1,\frac{m}{m-1}[$, satisfies the following estimates:
\begin{subequations}
  \begin{equation}\label{eq:215}
\| \bar{p}\|_{L^2(W_0^{1,r}(\Omega))}\leq c_p(\| w\|_{\mathbb{R}^{n_s}}+\| \sigma\|_{\mathbb{R}^{n_T}}),\quad\text{ for some }c_p>0.
\end{equation}
\begin{equation}\label{eq:215_mpr}
\| \bar{p}\|_{\mathbb{W}_0^2(W^{1,r}_0,W^{-1,r})}\leq \tilde c_p(\| w\|_{\mathbb{R}^{n_s}}+\| \sigma\|_{\mathbb{R}^{n_T}}),\quad\text{ for some }\tilde c_p>0.
\end{equation}
\end{subequations}
\end{remark}

\section{Bilevel optimal placement problem}
Our main goal consists in determining where new observation devices should be placed and when the measurements should be taken, to get new valuable information for the variational data assimilation process and obtain, as a consequence, a better reconstruction of the initial condition. As mentioned in the introduction, we tackle the aforementioned problem by considering a supervised learning approach in which we presuppose the existence of a training set $$\left(u_1^\dag,y_1^\dag\right),\ldots,\left(u_N^\dag,y_N^\dag\right),$$
consisting of clean reconstructions of the initial conditions (for instance, with exceptionally more observations) $u_i^\dag, ~i=1, \dots, N$ and the corresponding observations of the system state $y_i^\dag, ~i=1, \dots, N$. The idea of working with training sets is borrowed from machine learning and is a widespread practice nowadays. In our case, the information we want to learn is precisely the vector of optimal placements $w$ and optimal time subintervals $\sigma$ at which the measurements should be carried out.

\subsection{Problem statement}
To accomplish the mentioned goal, we consider a bilevel optimization approach where the lower-level instance is related to finding a solution to the semilinear variational data assimilation problem, while the upper-level one solves the optimal placement problem in time and space.

Mathematically, the starting mixed integer nonlinear programming formulation is the following:
\begin{subequations} \label{eq:301}
      \begin{equation} \label{eq:301a}
      \displaystyle
      \min_{w,\sigma\in\{0,1\}} \displaystyle\iint\limits_Q\sum_{j=1}^N L(y_j,y_j^\dag)~dxdt+\beta\int\limits_{\Omega}\sum_{j=1}^N l(u_j, u_j^\dag)~dx+\beta_w \sum_{k}w_k+\beta_\sigma \sum_{i}\sigma_i
      \end{equation}
subject to:
      \begin{equation} \label{eq:301b}
      \left\{
      \begin{array}{ll}
      \displaystyle
      \min_{u_j} \dfrac{1}{2}\displaystyle\int\limits_0^T\sum_{k,i} \left(w_k\sigma_i\rho_i(t)[y_j(x_k,t)-z_{oj}(x_k,t)]^2 \right) dt\\ \hspace{4cm} + \dfrac{1}{2} \|u_j-u_{bj}\|_{B^{-1}}+\frac{\vartheta}{2}\|\nabla(u_j-u_{bj})\|_{L^2(\Omega)}^2\\
       \text{subject to: } \\
       \begin{array}{rll}
      \dfrac{\partial y_j}{\partial t}+A y_j +g(y_j)&=&0 \qquad \text{ in }Q\\
      y_j&=&0 \qquad \text{ on }\Sigma\\
      y_j(0)&=&u_j \qquad \text{in }\Omega.
      \end{array}
      \end{array}
      \right.
      \end{equation}
\end{subequations}
for all $j=1,\hdots,N.$ Notice that the problem constraints correspond to $N$ training data assimilation problems, where $u_{bj}$ represents the background information for each assimilation instance. In addition, $\beta>0$ corresponds to a weighting term between the quality measures for the state and the initial condition and $\beta_w, \beta_\sigma>0$ are given penalization parameters. Additionally, $L$ and $l$ are loss functions and are assumed to verify the following conditions.

\begin{assumption}\label{assum:002}
For each $j=1,\hdots,N$,
\begin{enumerate}
    \item $L(y_j,y_j^\dag)=L(x,t,y_j,y_j^\dag):Q\times\mathbb{R}\times\mathbb{R}\mapsto\mathbb{R}$ is a measurable function with respect to $(x,t)\in Q$ for all $y_j$. It is twice differentiable with respect to $y_j$ for almost every $(x,t)\in Q$.
    \item $l(x,u_j,u_j^\dag):\Omega\times\mathbb{R}\times\mathbb{R}\mapsto\mathbb{R}$ is a measurable function with respect to $x\in \Omega$ for all $u_j$. It is twice differentiable with respect to $u_j$ for almost every $x\in \Omega$.
    \item $L$ and $l$ are assumed to be strongly convex with respect to $y_j$ and $u_j$, respectively.
    \end{enumerate}
\end{assumption}

Due to the difficulty of working with integer variables, we consider a relaxation of the binary variables. By impossing the natural box constraints $0\leq w\leq 1$ and $0\leq \sigma\leq 1$, we arrive at a new cost functional:
\begin{equation*}
\displaystyle
\min_{0\leq w,\sigma\leq1} J(\mathbf{y},\mathbf{w}):=\displaystyle\iint\limits_Q\sum_{j=1}^N L(y_j,y_j^\dag)~dxdt+\beta\int\limits_{\Omega}\sum_{j=1}^N l(u_j, u_j^\dag)~dx+\beta_w \displaystyle\sum_k w_k+\beta_\sigma\displaystyle\sum_i \sigma_i.
\end{equation*}

Further, replacing the variational data assimilation problem with its optimality system, we get the following Karush-Kuhn-Tucker (KKT) reformulation of the bilevel learning problem, for all $j=1,\ldots,N$:
\begin{subequations} \label{eq: bilevel problem}
      \begin{equation}\label{eq:401}
      \displaystyle
      \min_{0\leq w,\sigma \leq1} J(\mathbf{y},\mathbf{w})
      \end{equation}
      subject to:
      \begin{equation}\label{eq:400}
      \begin{array}{rll}
        \dfrac{\partial y_j}{\partial t}+A y_j +g(y_j)= &$0$& \text{ in }Q \\
        y_j= &$0$& \text{ on }\Sigma \\
        y_j(0) =& u_j & \text{ in }\Omega
      \end{array}
    \end{equation}
    \begin{equation}\label{eq:400_2}
    \begin{array}{rll}
        -\dfrac{\partial p_j}{\partial t} +A^* p_j +&g'(y_j) p_j \hspace{1cm} &\\ = &\displaystyle\sum_{k,i} w_k\sigma_i\rho_i(t)\left[ y_j(x,t)-z_{oj}(x,t)\right]\otimes\delta(x-x_k)  & \text{ in }Q\\
        p_j= &$0$& \text{ on }\Sigma \\
        p_j(T)= &$0$& \text{ in }\Omega
      \end{array}
    \end{equation}
    \begin{equation}\label{eq:400_3}
    \begin{array}{rll}
        -\vartheta\Delta(u_j-u_{bj})+B^{-1}(u_j-u_{bj})+p_j(0)= &$0$& \text{ in }\Omega\\
        u_j= &$0$& \text{ on }\Gamma
      \end{array}
      \end{equation}
\end{subequations}
strongly, very weakly and weakly, respectively.

Replacing the lower-level problem by its first-order optimality system is a common practice in bilevel programming (see, e.g., \cite{dempe2006foundations,Dempe2012}). If the lower-level problem is convex, the necessary condition is also sufficient, and the replacement may be fully justified. If that's not the case, there may be several stationary points that do not correspond to local minima of the bilevel problem, and additional criteria must be included in order to analyze and solve the upper-level instance.

In our case, due to the semilinearity in the dynamics, the solution of the lower-level problem is not necessarily unique. Therefore, we have to deal with multiplicity of stationary points for the data assimilation problem and, consequently, equations \eqref{eq:400}-\eqref{eq:400_3} may have several solutions. Differently from previous related contributions (see e.g. \cite{hollerkunisch2018}), where restrictive uniqueness and stability assumptions are made, we do not restrict the set of stationary points in advance, but analyze problem \eqref{eq: bilevel problem} as an optimization one constrained by a multi-state time-dependent system.

To this aim, let $n_s$ and $n_T$ be the total number of feasible spatial observation points and time subintervals, respectively. Hereafter, we will denote $\mathbf{w}=(w,\sigma)$ and $\mathbf{y}=(\mathbf{y}_1(\mathbf{w}),\ldots,\mathbf{y}_N(\mathbf{w}))$, where $\mathbf{y}_j(\mathbf{w})=(y_j(\mathbf{w}),p_j(\mathbf{w}),u_j(\mathbf{w}))$, for every $j=1,\ldots,N$.
In the next theorem existence of a solution to system \eqref{eq:400}-\eqref{eq:400_3}, for each $w\in\mathbb{R}^{n_s}$ and $\sigma\in\mathbb{R}^{n_T}$, is verified.

\begin{lemma}\label{teo:301}
For each $w\in\mathbb{R}^{n_s}$ and $\sigma\in\mathbb{R}^{n_T}$, the state system \eqref{eq:400}-\eqref{eq:400_3} has at least one solution
  $\mathbf{y}_j=(y_j,p_j,u_j)\in H^{2,1}(Q)\times L^2(W_0^{1,r}(\Omega))\times H_0^{1}(\Omega)$, for all $ j\in{1,\ldots,N},$ with $r\in [1,\frac{m}{m-1}[$.
\end{lemma}

\begin{proof}
The state system \eqref{eq:400}-\eqref{eq:400_3} corresponds to $N$ optimality systems for the data assimilation problem given by \eqref{eq:202_2}-\eqref{eq:207}. Existence of an optimal solution to \eqref{eq:202_2}-\eqref{eq:207} was proved in Theorem \ref{teo:201}. Consequently, there exists at least one solution to the state system.
\end{proof}

Setting $\mathbf{Y}:=H^{2,1}(Q)\times L^2(W_0^{1,r}(\Omega))\times H_0^{1}(\Omega)$ and $\mathbf{y}_j=(y_j,p_j,u_j)\in \mathbf{Y}$ solution of \eqref{eq:400}, an estimate in terms of the placement vectors $w$ and $\sigma$ is obtained.

\begin{lemma}\label{lem:estimate_w_s}
For all $ j\in{1,\ldots,N},$ let $\mathbf{y}_j=(y_j,p_j,u_j)\in \mathbf{Y}$ be a solution of equations \eqref{eq:400}-\eqref{eq:400_3}. Then, the following estimate holds
\[
\| \mathbf{y}\|_{\mathbf{Y}^N}\leq C_y(\| w\|_{\mathbb{R}^{n_s}}+\| \sigma\|_{\mathbb{R}^{n_T}}) + C_b,\qquad\text{ for some }C_y, C_b>0.
\]
\end{lemma}

\begin{proof}
Using the weak formulation of the gradient equation \eqref{eq:400_3} and testing it with $h_j=u_j\in H_0^{1}(\Omega)$, we get
      \begin{equation*}
      \displaystyle
      \|u_j\|_{B^{-1}}^2 +\vartheta\|\nabla u_j\|_{L^2(\Omega)}^2 \leq C \left( \|p_j(0)\|_{H^{-1}(\Omega)}+ \|u_{bj}\|_{H^{1}(\Omega)} \right) \|u_j\|_{H_0^1(\Omega)},
      \end{equation*}
for some $C>0$. Since the first norm is equivalent to the $L^2-$norm and applying the  reverse triangle inequality, we get that $\|u_j(\mathbf{w})\|_{H_0^{1}(\Omega)}\leq \|p_j(0)\|_{H^{-1}(\Omega)}+\|u_{bj}\|_{H^{1}(\Omega)}$.

On the other hand, using embedding results for interpolation spaces (see \cite[p.5]{amann2005nonautonomous} and \cite[p.185]{triebel1995interpolation}), it follows that $\mathbb{W}_0^2(W_0^{1,r},W^{-1,r})\hookrightarrow C([0,T];W^{2v-1,r}(\Omega))$, for all $0<v<\frac{1}{2}$. Moreover, taking $v=\frac{1}{2}-\epsilon$ and $n<r'\leq \frac{2n}{n-2+\epsilon}$ with $0<\epsilon<1$, it holds that $H^1(\Omega)\hookrightarrow W^{1-2v,r'}(\Omega)$ \cite[Theorem 8.3.4]{fattorini}, and consecuently $W^{2v-1,r}(\Omega)\hookrightarrow H^{-1}(\Omega)$. Therefore, $\mathbb{W}_0^2(W_0^{1,r},W^{-1,r})\hookrightarrow C([0,T];H^{-1}(\Omega))$. From the last result, it follows that
\[
\|p_j(0)\|_{H^{-1}(\Omega)} \leq C \|p_j\|_{\mathbb{W}_0^2(W_0^{1,r},W^{-1,r})},\quad\text{ for some }C>0.
\]
From here and using estimate \eqref{eq:215_mpr}, the following estimate for the control variable is verified:
      \begin{equation}\label{eq:est_u_w_s}
      \displaystyle
      \|u_j(\mathbf{w})\|_{H_0^{1}(\Omega)}\leq c_{pj}(\| w\|_{\mathbb{R}^{n_s}}+\|\sigma\|_{\mathbb{R}^{n_T}})+\|u_{bj}\|_{H^{1}(\Omega)}.
      \end{equation}

Furthermore, using estimate \eqref{eq3:003} and the inequality above, it holds that
\begin{equation}\label{eq:500}
  \| y_j(\mathbf{w})\|_{H^{2,1}(Q)}\leq c_{uj}\| u_j(\mathbf{w})\|_{H_0^1(\Omega)}\leq c_{yj}(\|w\|_{\mathbb{R}^{n_s}}+\|\sigma\|_{\mathbb{R}^{n_T}}) + c_{bj}.
\end{equation}

The result follows from estimates \eqref{eq:500}, \eqref{eq:215} and \eqref{eq:est_u_w_s}.
\end{proof}

\subsection{Existence of solution for the bilevel problem}
Next, we will show that problem \eqref{eq: bilevel problem} has at least one solution. To this aim, let  $V_{ad}:=\{\mathbf{w}=(w,\sigma)\in \mathbb{R}^{n_s}\times\mathbb{R}^{n_T}: 0\leq w\leq 1, 0\leq \sigma\leq 1\}$ be the set of admissible placement vectors.

\begin{theorem}\label{teo:502}
The minimization problem \eqref{eq: bilevel problem} has at least one solution $\bar{\mathbf{w}}=(\bar{w},\bar{\sigma})\in V_{ad}$ with $\bar{\mathbf{y}}=\mathbf{y}(\bar{\mathbf{w}})\in\mathbf{Y}^N$ its corresponding optimal state.
\end{theorem}

\begin{proof}
Since $V_{ad}$ is bounded, the functional $J$ is bounded from below and, therefore, has an infimum value $j:=\inf_{\mathbf{w}\in V_{ad}} J(\mathbf{y}, \mathbf{w})$. We may select a minimizing sequence $\{\mathbf{w}_n\}_{n\geq1}=\{(w_n,\sigma_n)\}_{n\geq1}$ such that $w_n\to\bar{w}$, $\sigma_n\to\bar{\sigma}$ as $n\to\infty$. Since $V_{ad}$ is closed and convex, $\bar{\mathbf{w}}=(\bar{w},\bar{\sigma})\in V_{ad}$.

For each $j=1,\ldots,N$, we will denote $\mathbf{y}_j(\mathbf{w}_n)=\mathbf{y}_n^j$. From Lemma \ref{lem:estimate_w_s}, $\{\mathbf{y}_n^j\}$ is bounded in $\mathbf{Y}$. Therefore, due to the reflexivity of $\mathbf{Y}$, there exists a subsequence, denoted the same, such that $\mathbf{y}_n^j \rightharpoonup \bar{\mathbf{y}}_j$ in $\mathbf{Y}$. Due to the compact embeddings $H^{2,1}(Q)\hookrightarrow\hookrightarrow L^\mu(Q), ~\mu \leq 10,$ and $W^{1,r}(\Omega) \hookrightarrow\hookrightarrow L^q(\Omega)$, for $q\leq \frac{mr}{m-r}$ and $r\in [1,\frac{m}{m-1}[$, it follows that $y_n^j\to\bar{y}_j$ in $L^\mu(Q)$ and $p_n^j \rightharpoonup \bar{p}_j$ in $L^2(Q)$.

Following a similar procedure as in Theorem \ref{teo:201}, jointly with the continuous differentiability of the nonlinearity, we get that $g(y_n^j)\to g(\bar{y}_j)$ and $g'(y_n^j)\to g'(\bar{y}_j)$ in $L^\mu(Q)$. This in turn allows us to pass to the limit in equations \eqref{eq:400}-\eqref{eq:400_3} and prove that for each $j=1,\ldots,N$, $\bar{\mathbf{y}}_j=\mathbf{y}_j(\mathbf{w})$ is indeed a solution of the system. Finally, the optimality of $\bar{\mathbf{w}}$ follows from the weakly lower semicontinuity of $J$. In fact,
\begin{align*}
j&=\lim\inf_{n\to\infty}\displaystyle\iint\limits_Q
\sum_{j=1}^N L(y_n^j,y_j^\dag)dxdt+\beta
\int\limits_{\Omega}\sum_{j=1}^N l(u^j_n, u_j^\dag)dx+\beta_w \displaystyle\sum_k \bar{w}_k+\beta_\sigma\sum_i \bar{\sigma}_i\\
&\geq \displaystyle\iint\limits_Q
\sum_{j=1}^N L(\bar{y}_j,y_j^\dag)dxdt+\beta
\int\limits_{\Omega}\sum_{j=1}^N l(\bar{u}_j, u_j^\dag)dx+\beta_w \displaystyle\sum_k \bar{w}_k+\beta_\sigma\sum_i \bar{\sigma}_i =J(\bar{\mathbf{y}},\bar{\mathbf{w}}),
\end{align*}
which concludes the proof.
\end{proof}

\subsection{Optimality system for the bilevel problem}
To formally derive the optimality system for the bilevel placement problem, we use the Lagrangian approach. For each $j=1,\ldots,N$, we consider the optimality system of the lower-level data assimilation problem, given by:
\begin{subequations}
      \begin{equation} \label{eq:OS-DA_1}
            \begin{aligned}
            \frac{\partial y_j}{\partial t} + A y_j +g(y_j)&=0 &&\text{in }Q,\\y_j&=0 &&\text{on }\Gamma,\\y_j(0)&=u_j && \text{in }\Omega.
            \end{aligned}
      \end{equation}
      \begin{multline} \label{eq:OS-DA_2}
            \displaystyle \iint\limits_Q\left(\frac{\partial\zeta_j}{\partial t}+A\zeta_j+g'(y_j)\zeta_j\right)p_j~dxdt+
            \int\limits_{\Omega}p_j(0) \zeta_j(0)dx \\
            -\iint\limits_Q\displaystyle\zeta_j\left(\sum_{k,i} w_k\sigma_i\rho_i(t)\left[y_j(x,t)-z_{oj}(x,t)\right] \delta(x-x_k)\right)dxdt=0,\forall\zeta_j\in H^{2,1}(Q),
      \end{multline}
      \begin{multline} \label{eq:OS-DA_3}
            \displaystyle\int\limits_{\Omega}(u_j(x)-u_{bj}(x))B^{-1}\tau_j ~dx +\vartheta\int\limits_{\Omega} \nabla(u_j-u_{bj}).\nabla \tau_j ~dx\\ + \int\limits_{\Omega}p_j(0)\tau_j ~dx=0,\forall\tau_j\in H_0^1(\Omega).
      \end{multline}
\end{subequations}
Formulations \eqref{eq:OS-DA_2} and \eqref{eq:OS-DA_3} correspond to the \emph{very weak} and the \emph{weak} forms of the adjoint and the gradient equation, respectively. Setting $\mathbf{\varrho}=(\mathbf{\varrho}_1,\ldots,\mathbf{\varrho}_N)$ as the Lagrange multiplier, where $\mathbf{\varrho}_j=(\eta_j,\varphi_j, \zeta_j,\tau_j)\in L^2(Q)\times H^{-1}(\Omega) \times H^{2,1}(Q)\times H_0^{1}(\Omega)$ for each $j=1,\ldots,N$, the Lagrangian functional is given by

      \begin{multline*}
      \mathcal{L}(\mathbf{y},\mathbf{w},\mathbf{\varrho})=J(\mathbf{y},\mathbf{w})+\sum\limits_{j=1}^N \left[ \iint\limits_Q \eta_j \left( \frac{\partial y_j}{\partial t}+ A y_j +g(y_j)\right)dx dt
      +\int\limits_{\Omega} \varphi_j \left(y_j(0)-u_j\right)dx \right.\\+ \iint\limits_Q\left(\frac{\partial\zeta_j}{\partial t} + A\zeta_j+g'(y_j)\zeta_j\right)p_j~dxdt+\int\limits_{\Omega}p_j(0)\zeta_j(0)dx\\
      - \iint\limits_Q\displaystyle\zeta_j\left(\sum_{k,i} w_k\sigma_i\rho_i(t)\left[y_j(x,t)-z_{oj}(x,t)\right]\delta(x-x_k)\right)dxdt\\ \left. +\displaystyle\int\limits_{\Omega}(u_j(x)-u_{bj}(x))B^{-1}\tau_j ~dx +\vartheta\int\limits_{\Omega} \nabla(u_j-u_{bj}).\nabla \tau_j ~dx + \int\limits_{\Omega}p_j(0)\tau_j ~dx \right].
      \end{multline*}
The bilevel adjoint system is obtained by setting the derivative of $\mathcal{L}(\mathbf{y},\mathbf{w},\mathbf{\varrho})$ with respect to $\mathbf{y}$ equal to zero. First, taking the derivative with respect to $y_j$, in direction $v_1^j$, we obtain
      \begin{multline*}
      \nabla_{y_j}\mathcal{L}(\mathbf{y},\mathbf{w},\mathbf{\varrho})(v_1^j)=\iint\limits_Q \nabla_{y_j}L(y_j) v_1^j~dx dt+ \iint\limits_Q \eta_j \left(\dfrac{\partial v_1^j}{\partial t}+A v_1^j + g'(y_j)v_1^j \right) dxdt\\
      +\int\limits_{\Omega} \varphi_j(0)v_1^j(0)~dx+ \displaystyle\iint\limits_Q \left( g''(y_j)p_j\zeta_j - \sum_{k,i} w_k\sigma_i\rho_i(t)\zeta_j(x,t)\delta(x-x_k)\right) v_1^j~dxdt=0,
      \end{multline*}
which implies that $\eta_j$ solves, in a \emph{very weak} sense, the following PDE:
      \begin{equation}\label{eq:307}
      \begin{array}{rll}
      \displaystyle
      -\dfrac{\partial \eta_j}{\partial t}+A^* \eta_j &+g'(y_j)\eta_j + g''(y_j)p_j\zeta_j &\\ &=  \displaystyle \sum_{k,i} w_k\sigma_i\rho_i(t)\zeta_j(x,t)\otimes\delta(x-x_k)
     -\nabla_{y_j}L(y_j)&\text{in }Q,\\
      \eta_j&=0&\text{on }\Sigma,\\
      \eta_j(T)& =0&\text{in }\Omega,
      \end{array}
      \end{equation}
and, additionally, $\varphi_j=\eta_j(0)$. Thanks to Assumption \ref{assum:001} and the regularity of $p_j \in L^2(W_0^{1,r}(\Omega))$ and $\zeta_j \in H^{2,1}(Q)$, the term $g''(y_j)p_j\zeta_j$ is actually well-defined as an element of $H^{2,1}(Q)^*$.

Now, taking the derivative with respect to $p_j$, in direction $v_2^j$, we get
      \begin{equation*}
\nabla_{p_j}\mathcal{L}(\mathbf{y},\mathbf{w},\mathbf{\varrho})(v_2^j)=\displaystyle\iint\limits_Q \left(\dfrac{\partial \zeta_j}{\partial t}+A \zeta_j+g'(y_j)\zeta_j\right) v_2^j~ dxdt+\displaystyle\int\limits_{\Omega}v_2^j(0) (\zeta_j(0)+\tau_j)dxdt=0
      \end{equation*}
and, consequently, $\zeta_j$ is the \emph{strong solution} of
      \begin{equation}\label{eq:308}
      \begin{array}{rll}
      \displaystyle
      \dfrac{\partial \zeta_j}{\partial t}+A \zeta_j + g'(y_j)\zeta_j =&$0$&\text{ in }Q,\\
      \zeta_j=&0&\text{ on }\Sigma,\\
      \zeta_j(0)=&-\tau_j&\text{ in }\Omega.
      \end{array}
      \end{equation}
Finally, taking the derivative with respect to $u_j$, in direction $v_3^j$,
      \begin{equation*}
      \nabla_{u_j}\mathcal{L}(\mathbf{y},\mathbf{w},\mathbf{\varrho})(v_3^j)=\displaystyle\int\limits_{\Omega}\left(\beta \nabla_{u_j} l(u_j)-\varphi_j\right)v_3^j~dx+\displaystyle\int\limits_{\Omega}v_3^jB^{-1}\tau_j ~dx +\vartheta\int\limits_{\Omega} \nabla v_3^j.\nabla \tau_j ~dx=0
      \end{equation*}
and, consecuently, $\tau_j$ is the unique \emph{weak} solution of
      \begin{equation}\label{eq:309}
        \begin{array}{rll}
        \displaystyle
        -\vartheta\Delta\tau_j +B^{-1}\tau_j =&\varphi_j-\beta \nabla_{u_j} l(u_j)&\text{ in }\Omega,\\
        \tau_j=&0&\text{ on }\Gamma.
        \end{array}
      \end{equation}
With this formal Lagrangian based derivation of the optimality system, we are able to state the main result of this section, whose rigourous proof is provided in Subsection \ref{subsection: existence of Lagrange mult}.

\begin{theorem} \label{thm: formal OS}
Let $(\bar{w},\bar{\sigma})\in\mathbb{R}^{n_s}\times\mathbb{R}^{n_T}$ be a local optimal solution to \eqref{eq: bilevel problem} with $\bar{\mathbf{y}}=\mathbf{y}(\bar{\mathbf{w}})\in\mathbf{Y}^N$ its corresponding optimal state. Then, there exists an adjoint state
$(\eta_j, \zeta_j, \tau_j) \in L^2(Q) \times H^{2,1}(Q) \times H_0^1(\Omega)$, for all $j=1,\ldots,N$, and Karush-Kuhn-Tucker multipliers $\lambda^a, \lambda^b \in \mathbb{R}^{n_s}\times\mathbb{R}^{n_T}$ satisfying, for all $j=1,\ldots,N$:

\noindent \underline{Adjoint system:}
\begin{subequations} \label{eq:310}
      \begin{equation}\label{eq:310a}
      \begin{aligned}
      \displaystyle
      -\dfrac{\partial \eta_j}{\partial t}+A^* \eta_j &+g'(y_j)\eta_j + g''(y_j)p_j\zeta_j &&\\=&  \displaystyle \sum_{k,i} w_k\sigma_i\rho_i(t)\zeta_j(x,t)\otimes\delta(x-x_k)
       -\nabla_{y_j}L(y_j)&&\text{ in }Q\\
      \eta_j=&0&&\text{on }\Sigma\\
      \eta_j(T)=&0&&\text{in }\Omega.
      \end{aligned}
      \end{equation}
      \begin{equation}\label{eq:310b}
      \begin{aligned}
      \displaystyle\dfrac{\partial \zeta_j}{\partial t}+A \zeta_j + g'(y_j)\zeta_j =&0&&\text{in }Q\\      \zeta_j=&0&&\text{on }\Sigma\\
      \zeta_j(0)=&-\tau_j&&\text{in }\Omega
      \end{aligned}
      \end{equation}
      \begin{equation}\label{eq:310c}
      \begin{aligned}
       \displaystyle-\vartheta\Delta\tau_j +B^{-1}\tau_j =&\eta_j(0)-\beta \nabla_{u_j} l(u_j)&\text{in }\Omega\\
      \tau_j=&0&\text{on }\Gamma,
      \end{aligned}
      \end{equation}
      very weakly, strongly and weakly, respectively,\\

      \noindent \underline{Gradient system:}
      \begin{multline}\label{eq:gradient_w}
            \beta_w-\displaystyle\sum\limits_{j=1}^N\int\limits_0^T\sum_i\sigma_i\rho_i(t)\zeta_j(x_k,t)\left(y_j(x_k,t)-z_{oj}(x_k,t)\right) ~dt\\ = \lambda^a_k- \lambda^b_k, \qquad \text{for all } k=1,\ldots,n_s,
      \end{multline}
      \begin{multline}\label{eq:gradient_sigma}
            \beta_\sigma-\displaystyle\sum\limits_{j=1}^N\int\limits_0^T\sum_k w_k\rho_i(t)\zeta_j(x_k,t)\left(y_j(x_k,t)-z_{oj}(x_k,t)\right)~dt\\= \lambda^a_{n_s+i}- \lambda^b_{n_s+i}, \qquad \text{for all } i=1\ldots,n_T,
      \end{multline}
      \underline{Complementarity system:}
      \begin{equation}\label{eq:Complementarity}
            \begin{aligned}
             &\lambda^a_{r}\geq0,\lambda^b_{r}\geq0,&& \hbox{for all $r=1,\ldots,n_s+n_T$} \\
              &\lambda^a_{k}\bar{w}_k=\lambda^b_k(\bar{w}_k-1)=0,&& \hbox{for all $k=1,\ldots,n_s$} \\
              &\lambda^a_{n_s+i}\bar{\sigma}_i=\lambda^b_{n_s+i}(\bar{\sigma}_i-1)=0,&& \hbox{for all $i=1,\ldots,n_T$} \\
             &0\leq \bar{w}_k\leq 1, && \hbox{for all $k=1,\ldots,n_s$}\\
             &0\leq \bar{\sigma}_i\leq 1, &&\hbox{for all $i=1,\ldots,n_T$}.
            \end{aligned}
      \end{equation}
\end{subequations}
\end{theorem}

\subsection{Existence of Langrange multipliers} \label{subsection: existence of Lagrange mult}
Since the constraint of the optimal placement problem is a multi-state system, we will prove existence of Lagrange multipliers for \eqref{eq: bilevel problem} (and therefore prove Theorem \ref{thm: formal OS}), by analyzing an adapted penalized version of the problem (see, e.g., \cite{bonnans1989optimal,lions1983controle}). For simplicity, we will restrict our attention to the case $j=1$, i.e, taking just one element in the training set. The extension to the case of a larger training set is straightforward.

Let us consider $(\bar{y},\bar{p},\bar{u})\in H^{2,1}(Q)\times L^2(W_0^{1,r}(\Omega)) \times H_0^1(\Omega)$ a solution of the optimization problem \eqref{eq: bilevel problem}, which exists thanks to Theorem \ref{teo:502}. The proposed adapted penalized problem consists in finding $(y_\gamma,p_\gamma,\mathbf{w}_\gamma,s_\gamma,v_\gamma) \in H^{2,1}(Q)\times L^2(W_0^{1,r}(\Omega))\times V_{ad}\times L^2(Q)\times L^2(Q)$, that solves:

\begin{subequations} \label{eq:penalized}
      \begin{multline}\label{eq:J_lambda}
      \min_{\mathbf{w},s,v} J_{\gamma}(y,p,\mathbf{w},s,v)= J(y,p,\mathbf{w})+\displaystyle \frac{\gamma}{2}\iint\limits_Q (s-g(y))^2 + \frac{\gamma}{2}\iint\limits_Q (v-g'(y)p)^2\\+ \frac{1}{2}\|p-\bar{p}\|^2_{L^2(Q)} + \frac{1}{2}\iint\limits_Q (s-g(\bar{y}))^2 + \frac{1}{2}\iint\limits_Q (v-g'(\bar{y})\bar{p})^2
      \end{multline}
      subject to:
      \begin{equation}\label{eq:400_pen}
      \begin{array}{rll}
        \dfrac{\partial y}{\partial t}+A y +s= &$0$& \text{ in }Q \\
        y= &$0$& \text{ on }\Sigma \\
        y(0) =& -G^{-1}p(0) & \text{ in }\Omega
      \end{array}
    \end{equation}
    \begin{equation}\label{eq:400_2_pen}
    \begin{array}{rll}
    -\dfrac{\partial p}{\partial t} +A^* p +v = &\displaystyle\sum_{k,i} w_k\sigma_i\rho_i(t)\left[ y(x,t)-z_{o}(x,t)\right]\otimes\delta(x-x_k)  & \text{ in }Q\\
        p= &$0$& \text{ on }\Sigma \\
        p(T)= &$0$& \text{ in }\Omega.
      \end{array}
    \end{equation}
    \end{subequations}
Here  $\gamma>0$ represents the penalization parameter and $G:H_0^1(\Omega)\rightarrow H^{-1}(\Omega)$ corresponds to the linear bijective solution operator arising from \eqref{eq:400_3}, i.e., for $u\in H_0^1(\Omega)$,
\[
\displaystyle\langle Gu,\tau\rangle_{H^{-1},H_0^1}:=\int\limits_{\Omega}(u-u_{b})B^{-1}\tau+\vartheta\int\limits_{\Omega} \nabla(u-u_{b}).\nabla \tau= -\int\limits_{\Omega}p(0)\tau ,\forall\tau\in H_0^1(\Omega).
\]

\begin{theorem}\label{teo:exists_penalized}
Let $(\bar{y},\bar{p},\bar{u},\bar{\mathbf{w}})\in H^{2,1}(Q)\times L^2(W_0^{1,r}(\Omega))\times H_0^1(\Omega)\times V_{ad}$ be a solution of \eqref{eq: bilevel problem}. The penalized problem \eqref{eq:penalized} has at least one solution $(y_\gamma,p_\gamma,\mathbf{w}_\gamma,s_\gamma,v_\gamma)\in H^{2,1}(Q)\times L^2(W_0^{1,r}(\Omega))\times V_{ad}\times L^2(Q)\times L^2(Q)$. Moreover, if $(y_\gamma,p_\gamma,\mathbf{w}_\gamma,s_\gamma,v_\gamma)$ is solution of \eqref{eq:penalized}, then there exist Lagrange multipliers $(\eta_\gamma,\zeta_\gamma) \in L^2(Q)\times H^{2,1}(Q)$ such that:

  \begin{subequations} \label{eq:adj_penalized}
        \begin{equation}\label{eq:adj_penalized_a}
        \begin{aligned}
        \displaystyle
  \displaystyle -\dfrac{\partial \eta_\gamma}{\partial t}+A^* \eta_\gamma  + \nabla_{y}L(y)-&\gamma g'(y_\gamma)(s_\gamma - g(y_\gamma)) &\\ =\gamma g''(y_\gamma)p_\gamma(v_\gamma - g'(y_\gamma)p_\gamma) & +\displaystyle\sum_{k,i} w_k\sigma_i\rho_i(t)\zeta_\gamma(x,t)\otimes\delta(x-x_k)&&\text{ in }Q\\
        \eta_\gamma=&0&&\text{on }\Sigma\\
        \eta_\gamma(T)=&0&&\text{in }\Omega,
        \end{aligned}
        \end{equation}
        \begin{equation}\label{eq:adj_penalized_b}
        \begin{aligned}
  \displaystyle\dfrac{\partial \zeta_\gamma}{\partial t}+A \zeta_\gamma =&\gamma g'(y_\gamma)(v_\gamma - g'(y_\gamma)p_\gamma) - (p_\gamma-\bar{p})&&\text{ in }Q\\      \zeta_\gamma=&0&&\text{on }\Sigma\\
        \zeta_\gamma(0)=&G^{-1}(\beta \nabla_{u}l(u)-\eta_\gamma(0))&&\text{in }\Omega,
        \end{aligned}
        \end{equation}
        very weakly and strongly, respectively,\\
      \begin{equation}\label{eq:penalization_v}
        \eta_\gamma + \gamma(s_\gamma - g(y_\gamma))+ (s_\gamma - g(\bar{y})) = 0
      \end{equation}
      \begin{equation}\label{eq:penalization_z}
        \zeta_\gamma + \gamma(v_\gamma - g'(y_\gamma)p_\gamma)+ (v_\gamma - g'(\bar{y})\bar{p}) = 0
      \end{equation}
      \begin{multline}\label{eq:gradient_pen_w}
            \left(\beta_w-\displaystyle\int\limits_0^T\sum_i\sigma_i\rho_i(t)\zeta_\gamma(x_k,t)\left(y_\gamma(x_k,t)-z_{o}(x_k,t)\right) ~dt\right)(\upsilon- (w_k)_\gamma)\geq0,\\\text{for all } \upsilon\in [0,1] \text{ and } k=1,\ldots,n_s,
      \end{multline}
      \begin{multline}\label{eq:gradient_pen_sigma}
            \left(\beta_\sigma-\displaystyle\int\limits_0^T\sum_k w_k\rho_i(t)\zeta_\gamma(x_k,t)\left(y_\gamma(x_k,t)-z_{o}(x_k,t)\right)~dt\right)(\upsilon - (\sigma_i)_\gamma)\geq0,\\\text{for all } \upsilon\in [0,1]\text{ and } i=1\ldots,n_T.
      \end{multline}
      \end{subequations}
\end{theorem}

\begin{proof}
Let $\{(y_n,p_n,\mathbf{w}_n,s_n,v_n)\}\subset H^{2,1}(Q)\times L^2(W_0^{1,r}(\Omega))\times V_{ad}\times L^2(Q)\times L^2(Q)$ be a minimizing sequence for problem \eqref{eq:penalized}. From the structure of $J_\gamma$, the minimizing sequence is bounded in $L^2(Q)\times L^2(Q) \times V_{ad}\times L^2(Q)\times L^2(Q)$. From the state and adjoint equations \eqref{eq:400_pen} and  \eqref{eq:400_2_pen}, it then follows that $\{y_n\}$ and $\{p_n\}$ are bounded in $H^{2,1}(Q)$ and $L^2(W_0^{1,r}(\Omega))$, respectively. Hence, there exists a weakly convergent subsequence, which will be denoted in the same way.

Taking such a weakly convergent subsequence and using the compactness of $H^{2,1}(Q)$ into $L^\mu(Q)$,  with $\mu\leq10$, it follows that $y_n\to\bar{y}$ strongly in $L^\mu(Q)$, and also $g(y_n)\to g(\bar{y})$ and $g'(y_n)\to g'(\bar{y})$ strongly in  $L^\mu(Q)$. Since $W^{1,r}(\Omega)\hookrightarrow\hookrightarrow L^q(\Omega)$ for $q\leq \frac{mr}{m-r}$, $p_n \rightharpoonup \bar{p}$ in $L^2(Q)$. Thereafter, proceeding as in Theorem \ref{teo:502} and from the lower semicontinuity of $J_\gamma$, we get that the limit point is optimal for problem  \eqref{eq:penalized}.

Existence of Lagrange multipliers $(\eta_\gamma,\zeta_\gamma)\in L^2(Q) \times H^{2,1}(Q)$, solution of the adjoint penalized system \eqref{eq:adj_penalized_a}-\eqref{eq:penalization_z}, follows directly from the linearity of the time-dependent system \eqref{eq:400_pen}-\eqref{eq:400_2_pen} (see, e.g., \cite{Lions}).
\end{proof}

\begin{lemma}\label{lem:convergence_penalized}
Let $\{(y_\gamma,p_\gamma,\mathbf{w}_\gamma,s_\gamma,v_\gamma)\}\subset H^{2,1}(Q)\times L^2(W_0^{1,r}(\Omega))\times V_{ad}\times L^2(Q)\times L^2(Q)$ be a sequence of solutions to the penalized problem \eqref{eq:penalized}. Then $\{(y_\gamma,p_\gamma,\mathbf{w}_\gamma,s_\gamma,v_\gamma)\}$ converges strongly in $H^{2,1}(Q)\times L^2(W_0^{1,r}(\Omega))\times V_{ad}\times L^2(Q)\times L^2(Q)$ to the solution $(\bar{y},\bar{p},\bar{\mathbf{w}},g(\bar{y}),g'(\bar{y})\bar{p})$.
\end{lemma}

\begin{proof}
Since the point $(\bar{y},\bar{p},\bar{\mathbf{w}},g(\bar{y}),g'(\bar{y})\bar{p})$ is feasible for \eqref{eq:penalized}, from the structure of the penalized cost functional, we get the bound
\[
J_\gamma(y_\gamma,p_\gamma,\mathbf{w}_\gamma,s_\gamma,v_\gamma)\leq J(\bar{y},\bar{p},\bar{\mathbf{w}}).
\]

Consequently,  $(y_\gamma,p_\gamma,\mathbf{w}_\gamma,s_\gamma,v_\gamma)$ is bounded in $L^2(Q)\times L^2(Q)\times V_{ad}\times L^2(Q)\times L^2(Q)$. From the state equation \eqref{eq:400_pen} and the adjoint equation  \eqref{eq:400_2_pen}, $\{y_\gamma\}$ and $\{ p_\gamma \}$
are bounded in $H^{2,1}(Q)$ and $L^2(W_0^{1,r}(\Omega))$, respectively, and, therefore, there exists a  subsequence, denoted  the same, such that
\[
(y_\gamma,p_\gamma,\mathbf{w}_\gamma,s_\gamma,v_\gamma) \rightharpoonup (y^*,p^*,\mathbf{w}^*,s^*,v^*)
\]
weakly in $H^{2,1}(Q)\times L^2(W_0^{1,r}(\Omega))\times \mathbb{R}^{n_s}\times\mathbb{R}^{n_T}\times L^2(Q)\times L^2(Q)$. Due to the compact embedding $H^{2,1}(Q)\hookrightarrow\hookrightarrow L^\mu(Q)$, with $\mu \leq 10$, the convergence of $y_\gamma\to y^*$ is strong in $L^\mu(Q)$, and also
$g'(y_\gamma)\to g'(y^*)$ in $L^\mu(Q)$. Since $W^{1,r}(\Omega)\hookrightarrow\hookrightarrow L^q(\Omega)$ for $q\leq \frac{mr}{m-r}$, $p_\gamma \rightharpoonup p^*$ weakly in $L^2(Q)$. The form of $J_\gamma$ implies that $\|s_\gamma-g(y_\gamma)\|_{L^2(Q)}\to 0$  and
$\|v_\gamma-g'(y_\gamma)p_\gamma\|_{L^2(Q)}\to 0$. Hence, $s^*-g(y^*)=0$ and $v^*-g'(y^*)p^*=0$. From the lower semicontinuity of $J$, it follows that

 \begin{align*}
J(&y^*,p^*,\mathbf{w}^*)+\frac{1}{2}\|p^*-\bar{p}\|_{L^2(Q)}^2 +\frac{1}{2}\|s^*-g(\bar{y})\|_{L^2(Q)}^2+\frac{1}{2}\|v^*-g'(\bar{y}))\bar{p}\|_{L^2(Q)}^2\\
   &\leq \lim\inf_{\gamma\to\infty} J(y_\gamma,p_\gamma,\mathbf{w}_\gamma) + \frac{1}{2}\|p_\gamma-\bar{p}\|_{L^2(Q)}^2 +\frac{1}{2}\|s_\gamma-g(\bar{y})\|_{L^2(Q)}^2+\frac{1}{2}\|v_\gamma-g'(\bar{y}))\bar{p}\|_{L^2(Q)}^2\\
   &\leq\lim\sup_{\gamma\to\infty} J_\gamma(y_\gamma,p_\gamma,\mathbf{w}_\gamma,s_\gamma,v_\gamma)\\
   &\leq J(\bar{y},\bar{p},\bar{\mathbf{w}}),
\end{align*}
which implies $y^*=\bar{y}$, $p^*=\bar{p}$, $s^*=g(\bar{y})$, and  $v^*=g'(\bar{y})\bar{p}$. The result follows from the uniqueness of the limits.
\end{proof}

\noindent \textbf{Proof of Theorem \ref{thm: formal OS}.}
Let $j=1$. We will prove that the subsequences $\{\zeta_\gamma\}$ and $\{ \eta_\gamma \}$ are bounded in $H^{2,1}(Q)$ and $L^2(Q)$, respectively. Notice that, from the properties of $g$, the right-hand side of equation \eqref{eq:adj_penalized_b} belongs to $L^2(Q)$. Indeed,
\begin{multline*}
\|\gamma g'(y_\gamma)(v_\gamma - g'(y_\gamma)p_\gamma) - (p_\gamma-\bar{p})\|_{L^2(Q)} \leq \gamma K\| v_\gamma - g'(y_\gamma)p_\gamma\|_{L^2(Q)} +\| p_\gamma-\bar{p}\|_{L^2(Q)}\\
\leq \gamma K\left(\| v_\gamma\|_{L^2(Q)} +K\|p_\gamma\|_{L^2(Q)}\right) +\| p_\gamma-\bar{p}\|_{L^2(Q)}.
\end{multline*}
From the form of the operator $G$, the initial condition of equation  \eqref{eq:adj_penalized_b} belongs to $H_0^1(\Omega)$. Therefore, using estimate \eqref{eq3:003}, the boundedness of $\{ \zeta_\gamma \}$  in $H^{2,1}(Q)$ is verified.

To prove the boundedness of $\{ \eta_\gamma \}$, we will show that the right-hand side of equation \eqref{eq:adj_penalized_a} belongs to $H^{2,1}(Q)^*$. Since $L^2(0,T;\mathcal{M}(\Omega))\hookrightarrow L^2(0,T;W^{-1,r}(\Omega))$, it follows that $\sum_{k,i} w_k\sigma_i\rho_i(t)\zeta_\gamma(x,t)\otimes\delta(x-x_k)$ belongs to $L^2(0,T;W^{-1,r}(\Omega)) \hookrightarrow H^{2,1}(Q)^*$.
For the second term, notice that $\nabla_{y}L(y)-\gamma g'(y_\gamma)(s_\gamma - g(y_\gamma))\in L^2(Q)$. Indeed,
\begin{multline*}
\|\nabla_{y}L(y)-\gamma g'(y_\gamma)(s_\gamma - g(y_\gamma))\|_{L^2(Q)}\leq \|\nabla_{y}L(y)\|_{L^2(Q)}+\|\gamma g'(y_\gamma)(s_\gamma - g(y_\gamma))\|_{L^2(Q)}\\
\leq\|\nabla_{y}L(y)\|_{L^2(Q)}+\gamma K\|s_\gamma - g(y_\gamma)\|_{L^2(Q)},
\end{multline*}
and, since $L^2(Q)\hookrightarrow H^{2,1}(Q)^*$, the result is verified for this component as well.
For the third term, notice that since $p_\gamma \in L^2(W_0^{1,r}(\Omega)) \hookrightarrow L^2(L^3(\Omega))$, the product $p_\gamma(v_\gamma - g'(y_\gamma)p_\gamma)$ belongs to $L^1(L^{6/5}(\Omega))$ (see e.g. \cite{hytonen2016analysis}) and
\begin{equation*}
  \|p_\gamma(v_\gamma - g'(y_\gamma)p_\gamma)\|_{L^1(L^{6/5}(\Omega))} \leq \|p_\gamma \|_{L^2(L^3(\Omega))} \|v_\gamma - g'(y_\gamma)p_\gamma \|_{L^2(Q)}.
\end{equation*}
Thanks to Assumption \ref{assum:001} and since $H^{2,1}(Q) \hookrightarrow L^\infty(L^6)$, it then follows that $g''(y_\gamma)p_\gamma(v_\gamma - g'(y_\gamma)p_\gamma)\in H^{2,1}(Q)^*$ with uniform bound.

Consequently, the sequence $\{\zeta_\gamma,\eta_\gamma\}$ is bounded in $H^{2,1}(Q)\times L^2(Q)$ and there exists a subsequence, denoted the same, and a limit point $(\bar{\zeta},\bar{\eta})\in H^{2,1}(Q)\times L^2(Q)$ such that $(\zeta_\gamma,\eta_\gamma) \rightharpoonup (\bar{\zeta},\bar{\eta})$ weakly in $H^{2,1}(Q) \times L^2(Q)$.
This jointly with the convergence result of Lemma \ref{lem:convergence_penalized} allow us to pass to the limit in \eqref{eq:adj_penalized} and obtain \eqref{eq:310a} and \eqref{eq:310b}. Equation \eqref{eq:310c} follows from the form of the operator $G$.
The Karush-Kuhn-Tucker conditions for the multipliers follow in a standard manner, by passing to the limit in the inequalities \eqref{eq:gradient_pen_w}-\eqref{eq:gradient_pen_sigma}.
\qed

\section{Numerical solution of the bilevel problem}
In this section we describe the algorithmic framework utilized for the construction of the proposed quasi-Newton solution algorithm.

\subsection{Sparsity enforcing penalty function}\label{sparsity}
Due to the kind of problem we deal with, the solutions that we would like to get are binary vectors that lead to the optimal sensors' spatial location and determine when the devices should be turned on. For that purpose, we consider the following modified objective functional:
  \begin{equation*}
  \displaystyle
  \min_{0\leq \mathbf{w}\leq1}J_{\epsilon}(\mathbf{y},\mathbf{w})=\displaystyle\iint\limits_Q\sum_{j=1}^N L(y_j,y_j^\dag)dxdt+\beta\int\limits_{\Omega}\sum_{j=1}^N l(u_j, u_j^\dag)dx+\beta_w\Phi_{\epsilon}(w)+\beta_\sigma\Phi_{\epsilon}(\sigma),
  \end{equation*}
where $\Phi_{\epsilon}(.)$, $\epsilon >0$, is the family of sparsity enforcing functions proposed in \cite[p.2135]{OED}:
      \[
      \begin{array}{ll}
        \Phi_{\epsilon}(x)=\displaystyle\sum_i \varphi_{\epsilon}(x_i) \text { and } &  \varphi_{\epsilon}(x_i)=\left\{
          \begin{array}{ll}
            \frac{x_i}{\epsilon}, & \hbox{$0\leq x_i\leq\frac{1}{2}\epsilon$} \\
            \pi_{\epsilon}(x_i), & \hbox{$\frac{1}{2}\epsilon< x_i\leq 2\epsilon$} \\
            1, & \hbox{$2\epsilon<x_i \leq1$}
          \end{array}
        \right.
      \end{array}
      \]
with $\pi_{\epsilon}(.)$ a uniquely defined polynomial of third order that makes $\varphi_{\epsilon}:[0,1]\longrightarrow[0,1]$ continuously differentiable. Setting  $\pi_{\epsilon}(x)=ax^3+bx^2+cx+e$, its coefficients can be obtained for each value of $\epsilon>0$ by solving the following system:
      \begin{equation*}\label{eq:32}
      \left(
      \begin{array}{cccc}
      \frac{\epsilon^3}{8}&\frac{\epsilon^2}{4}&\frac{\epsilon}{2}&1\\
      8\epsilon^3&4\epsilon^2&2\epsilon&1\\
      \frac{3\epsilon^2}{4}&\epsilon&1&0\\
      12\epsilon^2&4\epsilon&1&0
      \end{array}
      \right) \left(
      \begin{array}{c}
      a\\b\\c\\e
      \end{array}
      \right)= \left(
      \begin{array}{c}
      \frac{1}{2}\\1\\\frac{1}{\epsilon}\\0
      \end{array}
      \right)
    \end{equation*}

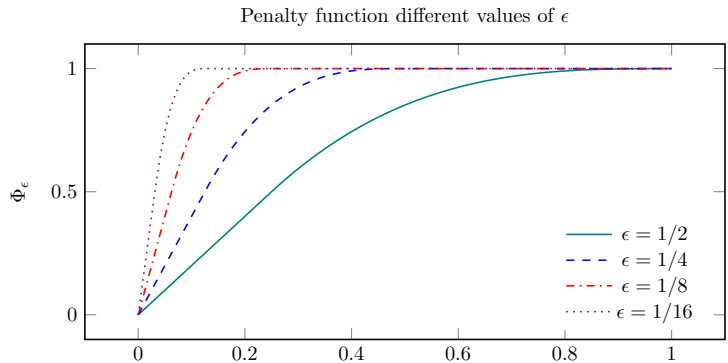
    \begin{figure}[ht]
    \centering
        \begin{tikzpicture}[baseline, every node/.style={scale=0.7}]
        \begin{axis}[
        title=Penalty function different values of $\epsilon$,
        ylabel=$\Phi_{\epsilon}$,
        line width=0.6pt,
        legend entries={$\epsilon=1/2$,$\epsilon=1/4$,$\epsilon=1/8$,$\epsilon=1/16$},
        legend style={draw=none},
        legend pos=south east,
        width=10cm,height=5.5cm]
        \addplot[color=teal] table[x=iter,y=f] {f_epsi1.dat};
        \addplot[color=blue,dashed] table[x=iter,y=f] {f_epsi2.dat};
        \addplot[color=red,dashdotted] table[x=iter,y=f] {f_epsi3.dat};
        \addplot[color=violet,dotted] table[x=iter,y=f] {f_epsi4.dat};
        \end{axis}
        \end{tikzpicture}%
      \end{figure}

\noindent Since we modified the objective function, the optimality system must be altered as well. However, the only change occurs in the computation of the gradient equation \eqref{eq:gradient_w}-\eqref{eq:gradient_sigma}, which takes the following form:
\begin{subequations} \label{eq:33}
\begin{equation}\label{eq:33_a}
\nabla F_{\epsilon}(\mathbf{w})_{k}=\beta_w \varphi_{\epsilon}'(w_k)-\displaystyle\sum\limits_{j=1}^N\int\limits_0^T\sum_i\sigma_i\rho_i(t)\zeta_j(x_k,t)\left(y_j(x_k,t)-z_{oj}(x_k,t)\right)dt,
\end{equation}
 for all $k=1,\ldots,n_s$, and
\begin{equation}\label{eq:33_b}
\nabla F_{\epsilon}(\mathbf{w})_{n_s+i} =\beta_\sigma \varphi_{\epsilon}'(\sigma_i)-\displaystyle\sum\limits_{j=1}^N\int\limits_0^T\sum_k w_k\rho_i(t)\zeta_j(x_k,t)\left(y_j(x_k,t)-z_{oj}(x_k,t)\right)dt,
\end{equation}
\end{subequations}

\noindent for all $i=1\ldots,n_T$. Defining the Karush-Kuhn-Tucker multipliers $\lambda^a =\max\{0,\nabla F_{\epsilon}(\bar{\mathbf{w}})\}$ and $\lambda^b =|\min\{0,\nabla F_{\epsilon}(\bar{\mathbf{w}})\}|$, the gradient equation \eqref{eq:33} at the optimal solution, may be rewritten as $\nabla F_{\epsilon}(\bar{\mathbf{w}})-\lambda^a+\lambda^b=0$, jointly with the complementarity conditions \eqref{eq:Complementarity}.

However, since $\Phi_{\epsilon}$ is a concave function, the solution to the bilevel problem may not be unique. To deal with this issue, we proceed as in \cite{Al}, i.e., for fixed parameters $\beta_w$ and $\beta_\sigma$, we solve the problem without sparsity enforcing penalty term. Then, we take the solution vectors $w$ and $\sigma$ as the initialization vectors for the problem with the nonconvex term. In practice, once $\epsilon$ is sufficiently small, $w$ and $\sigma$ will approach to binary vectors.

\subsection{Projected quasi-Newton methods}
In general, a projection method makes use of a descent direction of a unconstrained problem of the form,
$$\min_{w \in U_{ad}} f(w),$$
and, thereafter, projects the new iterate onto the feasible set $U_{ad}$, defined by the inequality box-constraints (\cite{JK}, pp.75). Thus, the update of the iterate $w_k$ is given by $$w_{k+1}=P_{U_{ad}}(w_k+\alpha_k d_k),$$ where $P_{U_{ad}}$ stands for the projection onto $U_{ad}$, $d_k$ is the descent direction, and $\alpha_k\in(0,1)$ a line search parameter. A modified Armijo rule is given for choosing the largest $\alpha_k$ such that
      \begin{equation}\label{eq:316}
      \displaystyle
      f\left(P_{U_{ad}}(w_k+\alpha_k d_k)\right)-f(w_k)\leq-\frac{\hat{\gamma}}{\alpha_k}\| P_{U_{ad}}(w_k+\alpha_k d_k)-w_k\|^2,
      \end{equation}
where $\hat{\gamma}\in(0,1)$.

Since the information provided by the Hessian matrix is not enough to generate descent directions for the constrained problem (see e.g. \cite[pp.98]{Kelley}), for the second order projected methods, it is necessary to use the reduced version of it, based on estimations of $\epsilon-$active and active sets:
      \[
      A^{\epsilon}(w)=\{i:a_i\leq w_i\leq a_i+\epsilon\text{ or } b_i\geq w_i\geq b_i-\epsilon\} \text{ and } A(w)=\{i:w_i=a_i \text{ or } w_i=b_i\}.
      \]
Here, $a_i$ and $b_i$ stand for the lower and upper bounds, respectively. The complements of these sets are  $I^{\epsilon}(w)$ and $I(w)$, they represent the $\epsilon-$inactive and inactive sets, respectively.

In a general way, if $S$ represents a generic index set, $R_S$ will denote the matrix $R_S=(\delta_{ij})$ if $i\in S$ or $j\in S$, with $\delta_{ij}=1$ if $i=j$ and $\delta_{ij}=0$ otherwise. With these notations, the reduced Hessian matrix, at an iterate $w_k$, is defined as:
      \begin{equation*}
      \begin{array}{ll}
      \tilde{R}(w_k,\epsilon_k,H_k)&=R_{A^{\epsilon_k}(w_k)}+R_{I^{\epsilon_k}(w_k)} H_k R_{I^{\epsilon_k}(w_k)}\\
      &\text{\textcolor[rgb]{1.00,1.00,1.00}{espacio en blanco}}\\
      &=\left\{
      \begin{aligned}
         &d_{ij},&&\text{if } i\in A^{\epsilon_k}(w_k) \text{ or } j\in A^{\epsilon_k}(w_k), \\
         &(H_k)_{ij},&&\text{ otherwise}.
      \end{aligned}
      \right.
      \end{array}
      \end{equation*}
Using the reduced matrix, the iteration of the projected method will be given by
      \begin{equation}\label{eq:318}
       w_{k+1}=P_{U_{ad}}(w_k-\alpha_k \tilde{R}(w_k,\epsilon_k,H_k)^{-1}\nabla f(w_k)).
      \end{equation}
If $H_k$ corresponds to the exact Hessian, \eqref{eq:318} represents the iterations of the projected Newton method, which converges q-quadratically upon identification of the active set \cite[Theorem 5.5.3]{Kelley}.

Now, let $H_k$ be an approximation of the Hessian matrix provided by the BFGS method. A possible update of the second term in the approximation of $\tilde{R}(w_{k+1},\epsilon_{k+1},H_{k+1})$ is given by (see \cite[p.102]{Kelley}):
      \[
      \tilde{H}_{k+1}=R_{I^{\epsilon_k}(w_k)}H_{k}R_{I^{\epsilon_k}(w_k)}-R_{I^{\epsilon_k}(w_k)}\frac{H_k s_k s_k^T H_k}{s_k^TH_k s_k}R_{I^{\epsilon_k}(w_k)}+\frac{y_k^{\#} (y_k^{\#})^T}{s_k^T y_k^{\#} },
      \]
where $y_k^{\#}=R_{I^{\epsilon_k}(w_k)}\left(\nabla f(w_{k+1})-\nabla f (w_k) \right)$.

Similarly to the unconstrained case, there is also an update to the inverse matrix of the BFGS projected method $B_k=H_k^{-1}$, which is given by
      \begin{equation}\label{eq:322}
        \tilde{B}_{k+1}=
      \left(I-\frac{s_k^{\#} (y_k^{\#})^T}{(y_k^{\#})^Ts_k^{\#}}\right)R_{I^{\epsilon_k}(w_k)} B_k R_{I^{\epsilon_k}(w_k)} \left(I-\frac{y_k^{\#} (s_k^{\#})^T}{( y_k^{\#})^Ts_k^{\#}}\right)+\frac{s_k^{\#} (s_k^{\#})^T}{(y_k^{\#})^Ts_k^{\#}},
      \end{equation}
where $s_k^{\#}=R_{I^{\epsilon_k}(w_k)}(w_{k+1}-w_{k})$. Therefore, the descent direction is given by
      \begin{equation}\label{eq:320}
      d_{k}=-R_{A^{\epsilon_k}(w_k)} \nabla f(w_k) - \tilde{B}_{k}\nabla f(w_k).
      \end{equation}

\subsection{Algorithm}
We propose a projected BFGS algorithm to solve the optimal observation placement problem \eqref{eq: bilevel problem}. What we aim to obtain is an optimal placement vector $\mathbf{w}=(w,\sigma)$ that allows us to obtain the best average reconstruction of each training pair $\left(u_j^\dag,y_j^\dag\right)$.

The algorithmic procedure consists of two main stages. In the first one we solve the location problem without considering the sparsity enforcing penalty function. To this end, we first set the parameter values, and then solve $N$ data assimilation problems given by \eqref{eq:301b}, by means of a standard BFGS method. Next, we have to compute the solution of the corresponding $N$ adjoint systems to the bilevel problem. Notice that these steps of the algorithm can be computed in parallel, since for each $j=1,\ldots, N$ the systems are independent from each other. Afterwards, we compute the gradient of the bilevel problem, estimate the $\epsilon-$active sets, and finally, using the BFGS projected method, calculate the placement vector update.

The second stage of our algorithm consists in using as initialization parameters the location vectors obtained in the previous stage and repeat the process described above to solve the placement problem which includes the sparsity enforcing penalty function. Since the initialization vector is already close to a sparse solution, the second stage of the algorithm should take just few iterations to get the result.

Concerning the line search, in both stages, we require that the  parameter $\alpha_k$ belongs to the following set
      \[\displaystyle\left\{\frac{1}{2^i\|\nabla F(\mathbf{w_0}) \|}, i\in\{0,1,2,\ldots\}\right\},\]
where $\nabla F(\mathbf{w_0})$ represents the gradient of the given initial placement vector $\mathbf{w_0}$. We summarize the steps described above in Algorithm 1.

 \begin{table}[ht]
 \label{alg:301}
 \begin{tabular}{@{}l@{}}
 \toprule
 \textbf{Algorithm 1}          \\ \midrule
 \textbf{Inputs:} $m$, $n$, $\beta$, $\beta_w$, and  $\beta_\sigma$ (problem dimension, parameters)\\
 \hspace{0.3cm} First stage:\\
  1. Set $\mathbf{w_0}=(w_0,\sigma_0)$, and $k=0$. Compute $\nabla F(\mathbf{w_0})$.\\
  2. \textbf{ Repeat} \\
  3. \hspace{0.5cm}For each $j=1,\ldots,N$, find $(y_{k_j},p_{k_j},u_{k_j})$ solution of \eqref{eq:400}.\\
  4. \hspace{0.5cm}For each $j=1,\ldots,N$, find $(\eta_{k_j},\zeta_{k_j})$ solution of the system \eqref{eq:310a}-\eqref{eq:310c}.\\
  5. \hspace{0.5cm}Find $\nabla F(\mathbf{w_k})$ using \eqref{eq:gradient_w} and \eqref{eq:gradient_sigma}.\\
  6. \hspace{0.5cm}Estimate the $\epsilon-$active set of $(w_k,\sigma_k)$.\\
  7. \hspace{0.5cm}Compute $\tilde{B}_{k+1}$ according to \eqref{eq:322}.\\
  8. \hspace{0.5cm}Find a descent direction $d_{k}$ using \eqref{eq:320}.\\
  9. \hspace{0.5cm}Find $\alpha_k \in \left\{\frac{1}{2^i\|\nabla F(\mathbf{w_0}) \|},  i\in\{0,1,2,\ldots\} \right\}$ that verifies \eqref{eq:316}. \\
  10. \hspace{0.3cm}Update $w_{k+1}=P_{V_{ad}}(\mathbf{w_k}+\alpha_k d_k)$\\
  11. \hspace{0.3cm}Set $k=k+1$.\\
  12. \textbf{Until} stopping criteria \\
 \hspace{0.5cm} Second stage:\\
  13. Set $\mathbf{w_0}=(w,\sigma)$, and $k=0$. Compute $\nabla F_{\epsilon}(\mathbf{w_0})$.\\
  14. Repeat the steps but considering \eqref{eq:33_a}-\eqref{eq:33_b} instead of \eqref{eq:gradient_w}-\eqref{eq:gradient_sigma}.\\
 \textbf{Outputs:} $w$ and $\sigma$ (optimal placement vectors)\\
 \bottomrule
 \end{tabular}
\end{table}

Intuitively, if the iterations of the projected BFGS start close to a non-degenerate local minimum together with a good approximation of the Hessian, it is expected that the iterations of the method will converge q-superlinearly, as will be experimentally verified in the next section.

\section{Computational experiments}\label{num}
In this section we report on the computational results obtained by using Algorithm 1. The considered spatial domain is the unit square $\Omega=]0,1[\times]0,1[$, while the time domain is the interval $]0,T[=]0,1[$. For each $j=1,\ldots,N$, we will work with $g(y_j)=\frac{y_j}{\sqrt{y_j^2 + \varepsilon ^2 }}$, $\varepsilon =\frac{1}{10}$, as the nonlinear term in the data assimilation problem. This nonlinearity arises from a $\mathcal{C}^\infty$ regularization of the absolute value function. We consider in the right-hand side of the data assimilation dynamical system a function that vanishes in time, namely $f(x,t)=(T-t)\sin (\pi x)$. We use a finite differences discretization scheme in space and an implicit Euler method in time, with spatial and time discretization steps $h=1/(m-1)$ and $\tau=1/(n+1)$, respectively. At each iteration of the implicit scheme, we use Newton's method to solve the nonlinear system.

In  what follows, we consider $L(y_j,y_j^\dag)=(y_j-y_j^\dag)^2$ and $l(u_j,u_j^\dag)=(u_j-u_j^\dag)^2$ for all $j=1,\ldots,N$, in the loss functionals of the optimal placement problem.

The numerical experiments are divided in two sets. In Subsection \ref{num}.1, we provide a numerical study of our method considering just one element of the training set. We aim to observe how the solution vectors vary when we use different penalization parameters. Thereafter, in Subsection \ref{num}.2, we use a larger training set and focus on learning the optimal placement vectors' structure.

For all the experiments, we set $n=12$, i.e., $n+2=14$ time subintervals, with $t_i=i\tau$, for all $i=1,\ldots, n+2$, and $m=20$, yielding $m^2=400$ possible points inside the unit square where the location of a sensor is possible. The background error covariance matrix $B^{-1}$ will be set as $\alpha I$, where $I$ is the identity operator and $\alpha$ takes the value of $1\times 10^{-1}$.

Since the sparsity enforcing penalty term is just an approximation of the counting norm, which indicates the number of non-zero entries of a vector, it is still possible to get values between zero and one. In those cases, for a better classification, we divide the interval $[0,1]$ into subintervals and count the number of elements that belong to each one. For the placement vector $w$, we fix the following intervals: $\mathbf{I_1}w=]0,0.2]$, $\mathbf{I_2}w=]0.2,0.8]$ and $\mathbf{I_3}w=]0.8,1[$.  For the time subintervals placement vector $\sigma$, we consider the following subintervals: $\mathbf{I_1}\sigma=]0,0.25[$, $\mathbf{I_2}\sigma=]0.25,0.50[$, $\mathbf{I_3}\sigma=[0.5,0.75[$, and $\mathbf{I_4}\sigma=[0.75,1[$.
\subsection{Single training pair experiments}
In this first set of experiments, we compute the training state $y^\dag$ by solving the variational data assimilation problem using the initial condition $u^\dag(x,y)=\sin(2\pi x)\sin(2\pi y)$, see Figure \ref{fig:ic}.
\begin{figure}[ht]
\centering
\input{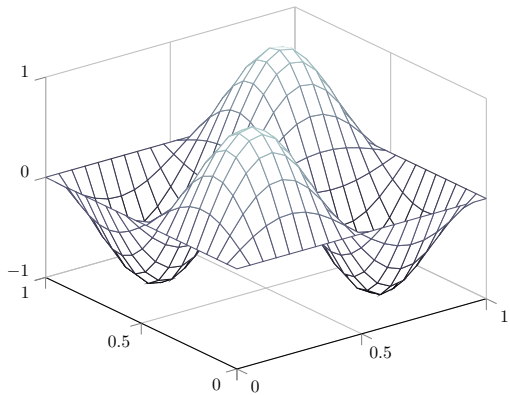}
\caption{Initial condition $u^\dag(x,y)$.}\label{fig:ic}
\end{figure}

To illustrate the behavior of the lower-level problem, for the given initial condition, we provide two snapshots of the model problem's evolution, see Figure \ref{fig:00}. Additionally, in the data assimilation problem, we add Gaussian noise to the observed state $z_o$, in order to test the robustness of the approach.
\begin{figure}[ht]
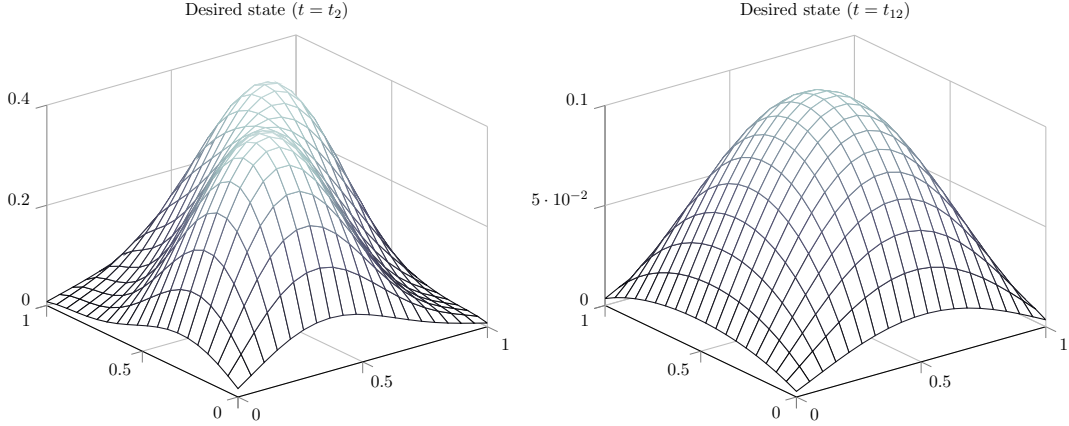

\centering
\input{y2_m20.tex}
\input{y12_m20.tex}
\caption{Snapshots of the state at times $t\approx 0.1429$ (left) and $t\approx 0.8579$ (right).}\label{fig:00}
\end{figure}

\subsubsection{Experiment 1}
The structure of the optimal placement vector, both in time and space, depends on the value of the penalization parameters $\beta_w$ (for the location vector), $\beta_\sigma$ (for the time vector), and $\beta$. The main goal in this experiment is to observe how $\mathbf{w}=(w,\sigma)$ changes its structure when we work with different values of these parameters. We perform the experiments in two stages. In the first one, we fix the parameters  $\beta_\sigma$ and $\beta$, and let $\beta_w$ be the one that changes. In the second stage, we vary the values of $\beta_\sigma$, while keeping the values of $\beta_w$ and $\beta$ fix.

\underline{First setting:} $\beta_\sigma=0$ and $\beta=0.09$, with sparsity enforcing penalty function $\Phi_{\epsilon}$ taking  $\epsilon=\frac{1}{2}$. Table \ref{tab:301} contains the information about the number of entries of the vectors $w$ and $\sigma$ that take the value zero and one. As expected, when the penalization parameter $\beta_w$ increases, the decision vector becomes sparser, with no elements between 0 and 1.

\begin{small}
\begin{table}[ht]
\centering
\begin{tabular}{@{}cccccc@{}}
\toprule
$\beta_w$ &$\#$ zeros in w &$\#$ ones in w &$\#$ zeros in $\sigma$ & $\#$ ones in $\sigma$ \\
\midrule
$1\times 10^{-5}$&0&400&0&14\\
0.0001&88&312&0&14\\
0.0003&96&304&0&14\\
0.0005&116&284&0&14\\
0.0010&156&244&0&14\\
0.0020&176&224&0&14\\
0.0030&200&200&0&14\\
0.0050&228&172&0&14\\
0.0060&236&164&0&14\\
0.0070&260&140&0&14\\
0.0072&268&132&0&14\\
0.0073&302&98&0&14\\
0.0074&334&66&0&14\\
0.0078&357&43&0&14\\
0.0079&361&39&0&14\\
0.0080&400&0&14&0\\
\bottomrule
\end{tabular}
\caption{Experiment 1. Setting 1 - Changes in $w$'s structure with different values of $\beta_w$}\label{tab:301}
\end{table}
\end{small}

In Table \ref{tab:301-1}, $\mathbf{J_0}$ represents the value of the cost functional at the beginning of the first stage of the algorithm, while  $\mathbf{J_{\text{end}}}$ denotes the value of the cost functional at the end of the second stage. Likewise, \textbf{iter} corresponds to the total number of projected quasi-Newton steps, performed in both stages of the algorithm, needed to get the location vectors. On each iteration of the upper-level problem, the algorithm solves the variational data assimilation problem, we report on \textbf{iter DA} the number of inner data assimilation iterations. The total number of forward and backward partial differential systems solved by Algorithm 1 in steps 3 and 4 is reported in the variable \textbf{PDE-solved}. We also show the values of the  location vectors' norms, which are given by $\|w\|_{\ell_1}=\sum_k w_k$ and $\|\sigma\|_{\ell_1}=\sum_i \sigma_i$.
\begin{small}
\begin{table}[ht]
\centering
\begin{tabular}{@{}ccccccccc@{}}
\toprule
$\beta_w$ &$\mathbf{J_0}$ &
$\mathbf{J_{\text{end}}}$ &$\mathbf{\| w \|}$ &$\mathbf{\| \sigma\|}$&\textbf{iter}&\textbf{iter DA}&\textbf{PDE-solved}\\ \midrule
$1\times 10^{-5}$&0.025968&0.025916&400&14&8&20&16\\
0.0001&0.061968&0.053177&312&14&7&20&14\\
0.0003&0.14197&0.11318&304&14&5&20&10\\
0.0005&0.22197&0.16398&284&14&5&20&10\\
0.0010&0.42197&0.26598&244&14&5&20&10\\
0.0020&0.82197&0.47000&224&14&5&20&10\\
0.0030&1.222&0.62201&200&14&5&20&10\\
0.0050&2.022&0.88204&172&14&5&20&10\\
0.0060&2.422&1.0060&164&14&5&20&10\\
0.0070&2.822&1.0021&140&14&5&20&10\\
0.0072&2.902&0.9725&132&14&5&20&10\\
0.0073&2.942&0.7370&98&14&15&20&30\\
0.0074&2.982&0.4924&66&14&15&20&30\\
0.0078&3.142&0.5437&43&14&8&20&16\\
0.0079&3.182&0.3621&39&14&8&20&16\\
0.0080&3.222&0.0225&0&0&15&17&30\\
\bottomrule
\end{tabular}
\caption{Experiment 1. Setting 1 - Decreasing of $\|w\|$ for different values of $\beta_w$}\label{tab:301-1}
\end{table}
\end{small}

The structure of the optimal placement vector $w$ can be visualized in Figure \ref{fig:306}, for five different increasing values of $\beta_w$. Apart from the increasing sparsity, the bilevel criterion favors the location of the observations in four well-defined spots of the domain, where the training initial condition actually plays a more relevant role. For this choice of parameters, we obtain both solution vectors $w$ and $\sigma$ as binary ones. There is no need of postprocessing.

  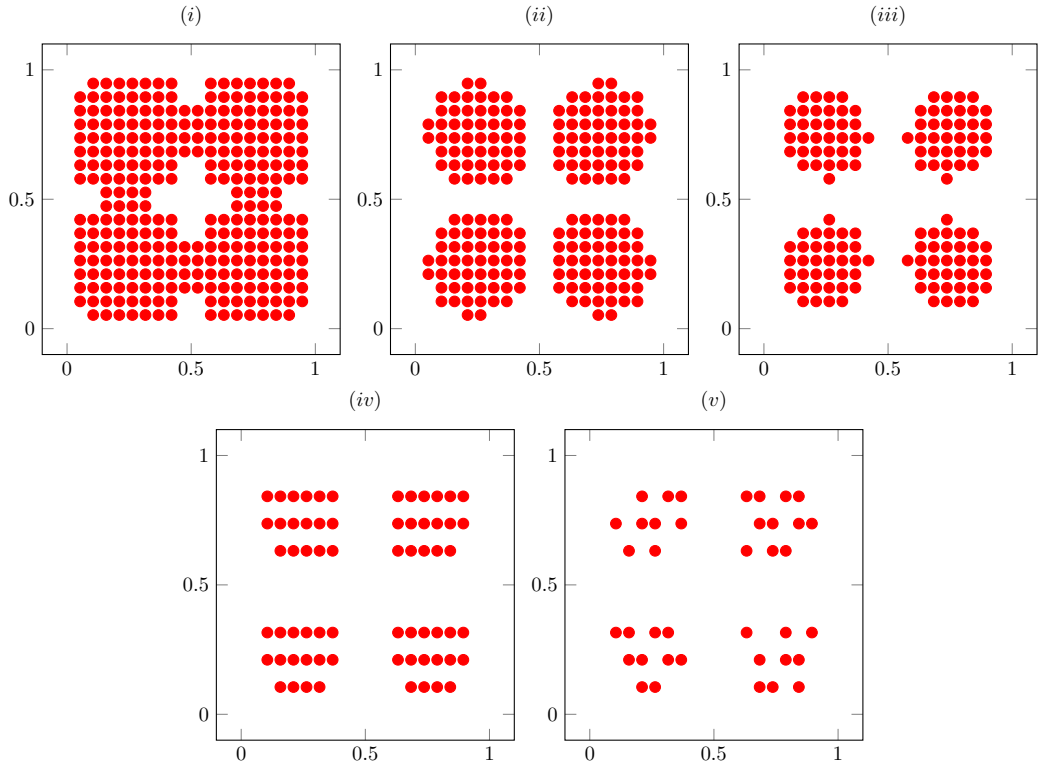
\begin{figure}[ht]
  \begin{tikzpicture}[baseline,  every node/.style={scale=0.7}]
  \begin{axis}[xmin=-0.1,xmax=1.1,ymin=-0.1,ymax=1.1,
  title=$(i)$,
  legend style={at={(0.5,-0.15)},anchor=north,legend columns=-1},width=5.5cm,height=5.7cm]
  \addplot[
  scatter, only marks, point meta=explicit symbolic, scatter/classes={
  1={mark=*,red,style={mark size=2pt}},%
  3={mark=*, blue,style={mark size=1.2pt}},
  2={mark=*,teal,style={mark size=0.7pt}},
  0={mark=*, violet,style={mark size=0.3pt}}},
  ]
  table[meta=label] {exp1_w1_m20.dat};
  \end{axis}
  \end{tikzpicture}
  \begin{tikzpicture}[baseline, every node/.style={scale=0.7}]
  \begin{axis}[xmin=-0.1,xmax=1.1,ymin=-0.1,ymax=1.1,
   title=$(ii)$,
   legend style={at={(0.5,-0.15)},anchor=north,legend columns=-1},width=5.5cm,height=5.7cm]
  \addplot[
  scatter, only marks, point meta=explicit symbolic, scatter/classes={
  1={mark=*,red,style={mark size=2pt}},%
  3={mark=*, blue,style={mark size=1.2pt}},
  2={mark=*,teal,style={mark size=0.7pt}},
  0={mark=*, violet,style={mark size=0.3pt}}},
  ]
  table[meta=label] {exp1_w2_m20.dat};
  \end{axis}
  \end{tikzpicture}
   \begin{tikzpicture}[baseline, every node/.style={scale=0.7}]
   \begin{axis}[xmin=-0.1,xmax=1.1,ymin=-0.1,ymax=1.1,
    title=$(iii)$,
    legend style={at={(0.5,-0.15)},anchor=north,legend columns=-1},width=5.5cm,height=5.7cm]
   \addplot[
   scatter, only marks, point meta=explicit symbolic, scatter/classes={
   1={mark=*,red,style={mark size=2pt}},%
   3={mark=*, blue,style={mark size=1.2pt}},
   2={mark=*,teal,style={mark size=0.7pt}},
   0={mark=*, violet,style={mark size=0.3pt}}},
   ]
   table[meta=label] {exp1_w3_m20.dat};
   \end{axis}
   \end{tikzpicture}
   \begin{tikzpicture}[baseline, every node/.style={scale=0.7}]
   \begin{axis}[xmin=-0.1,xmax=1.1,ymin=-0.1,ymax=1.1,
    title=$(iv)$,
    legend style={at={(0.5,-0.15)},anchor=north,legend columns=-1},width=5.5cm,height=5.7cm]
   \addplot[
   scatter, only marks, point meta=explicit symbolic, scatter/classes={
   1={mark=*,red,style={mark size=2pt}},%
   3={mark=*, blue,style={mark size=1.2pt}},
   2={mark=*,teal,style={mark size=0.7pt}},
   0={mark=*, violet,style={mark size=0.3pt}}},
   ]
    table[meta=label] {exp1_w5_m20.dat};
   \end{axis}
   \end{tikzpicture}
   \begin{tikzpicture}[baseline, every node/.style={scale=0.7}]
   \begin{axis}[xmin=-0.1,xmax=1.1,ymin=-0.1,ymax=1.1,
    title=$(v)$,
    legend style={at={(0.5,-0.15)},anchor=north,legend columns=-1},width=5.5cm,height=5.7cm]
   \addplot[
   scatter, only marks, point meta=explicit symbolic, scatter/classes={
   1={mark=*,red,style={mark size=2pt}},%
   3={mark=*, blue,style={mark size=1.2pt}},
   2={mark=*,teal,style={mark size=0.7pt}},
   0={mark=*, violet,style={mark size=0.3pt}}},
   ]
    table[meta=label] {exp1_w6_m20.dat};
   \end{axis}
   \end{tikzpicture}
  \caption{Optimal placement vector's structure. Different values of $\beta_w$. $(i) \beta_w=0.0005$, $(ii) \beta_w=0.0030$, $(iii) \beta_w=0.0070$, $(iv) \beta_w=0.0074$,$(v) \beta_w=0.0079$.}\label{fig:306}
  \end{figure}

\underline{Second setting}:  Choosing $m=10$, yielding $m^2=100 $ total points in the spatial domain, setting $\beta_w=0.0001$ and $\beta=0.1$, and letting $\beta_\sigma$ be the one that varies, we obtain changes in the structure of the time subintervals vector $\sigma$. For the sparsity enforcing penalty function $\Phi_\epsilon$, we choose  $\epsilon=\frac{1}{8}$. As the penalization parameter $\beta_\sigma$ increases, the structure of $\sigma$ becomes sparser. We report these changes in Table \ref{tab:303}.

\begin{small}
\begin{table}[ht]
\centering
\begin{tabular}{@{}cccccc@{}}
\toprule
$\beta_\sigma$ &$\#$ zeros in $\sigma$ & $\#$ ones in $\sigma$ &$\#$ zeros in w &$\#$ ones in w\\
\midrule
0.001&0&14&0&100\\
0.002&0&14&0&100\\
0.003&0&14&0&100\\
0.005&0&14&0&100\\
0.006&11&3&36&64\\
0.007&11&3&36&64\\
0.008&11&3&36&64\\
0.009&11&3&36&64\\
0.01 &11&3&36&64\\
0.02 &11&3&36&64\\
0.05 &14&0&100&0\\
\bottomrule
\end{tabular}
\caption{Experiment 1. Setting 2. Changes in $\sigma$'s structure with different values of $\beta_\sigma$}\label{tab:303}
\end{table}
\end{small}

In Table \ref{tab:303-1} we show the total number of iterations that the algorithm requires to reach the solution. As in the previous experiment, the norm of the time vector decreases as the penalization parameter increases. In Figure \ref{fig:305}, we can observe the structure of the time subintervals when the sensors/devices have to be turned on or turned off according to the chosen parameters. Enforcing sparsity in the time variable can be used to develop strategies about when the sensors/devices should be turned on or off. Doing this could be useful especially if the devices are energy demanding or need time to properly start to work.
\begin{small}
\begin{table}[ht]
\centering
\begin{tabular}{@{}cccccccc@{}}
\toprule
$\beta_\sigma$ &$\mathbf{J_0}$ &
$\mathbf{J_{\text{end}}}$ &$\mathbf{\| w \|}$ &$\mathbf{\| \sigma\|}$&\textbf{iter} &\textbf{iter DA}&\textbf{PDE-solved}\\ \midrule
0.001&0.043601&0.043599&100&14&6&19&12\\
0.002&0.057601&0.057592&100&14&6&19&12\\
0.003&0.071601&0.071116&100&14&23&19&46\\
0.005&0.099601&0.09786&100&14&23&19&46\\
0.006&0.1136&0.04446&64&3&18&16&36\\
0.007&0.1276&0.047456&64&3&9&16&18\\
0.008&0.1416&0.050456&64&3&9&16&18\\
0.009&0.1556&0.053456&64&3&9&16&18\\
0.01 &0.1696&0.056456&64&3&8&16&16\\
0.02 &0.3096&0.08646&64&3&8&16&16\\
0.05 &0.7296&0.17646&0&0&8&16&16\\
\bottomrule
\end{tabular}
\caption{Experiment 1 - Setting 2. Decreasing of $\|\sigma\|$ for different values of $\beta_\sigma$}\label{tab:303-1}
\end{table}
\end{small}

      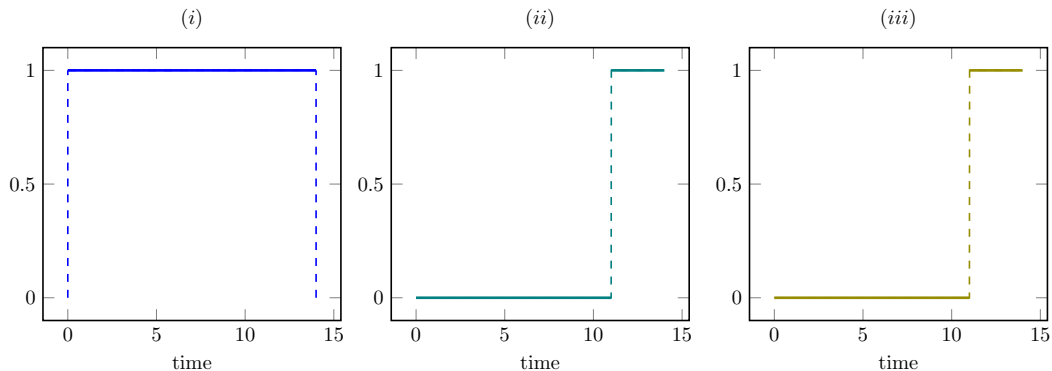
\begin{figure}[ht]
      \centering
      \begin{tikzpicture}[baseline, trim left=0.2cm,every node/.style={scale=0.7}]
      \begin{axis}[ymin=-0.1,ymax=1.1,
      title=$(i)$,
      xlabel=time,
      line width=0.6pt,
      width=5.5cm,height=5.2cm]
      \addplot[color=blue,dashed] table[x=t,y=f_t] {exp1_sigma1.dat};
      \addplot[color=blue,line width=1pt]coordinates{(0, 1)(14, 1)};
      \addplot[color=blue,dashed]coordinates{(0, 0)(0, 1)};
      \addplot[color=blue,dashed]coordinates{(14, 0)(14, 1)};
      \end{axis}
      \end{tikzpicture}%
      \begin{tikzpicture}[baseline,every node/.style={scale=0.7}]
      \begin{axis}[ymin=-0.1,ymax=1.1,
      title=$(ii)$,
      xlabel=time,
      line width=0.6pt,
      width=5.5cm,height=5.2cm]
      \addplot[color=teal,dashed] table[x=t,y=f_t] {exp1_sigma2.dat};
      \addplot[color=teal,line width=1pt]coordinates{(0, 0)(11, 0)};
      \addplot[color=teal,line width=1pt]coordinates{(11, 1)(14, 1)};
      \end{axis}
      \end{tikzpicture}
      \begin{tikzpicture}[baseline,every node/.style={scale=0.7}]
      \begin{axis}[ymin=-0.1,ymax=1.1,
      title=$(iii)$,
      xlabel=time,
      line width=0.6pt,
      width=5.5cm,height=5.2cm]
      \addplot[color=olive,dashed] table[x=t,y=f_t] {exp1_sigma2.dat};
      \addplot[color=olive,line width=1pt]coordinates{(0, 0)(11, 0)};
      \addplot[color=olive,line width=1pt]coordinates{(11, 1)(14, 1)};
      \end{axis}
      \end{tikzpicture}
      \caption{Optimal time invervals' structure. Different values of $\beta_\sigma$. $(i) \beta_\sigma=0.002$, $(ii) \beta_\sigma=0.008$, $(iii) \beta_\sigma=0.02$.}\label{fig:305}
      \end{figure}

\subsubsection{Experiment 2}
In many real situations, there are places where locating sensors could be difficult or expensive. Therefore, not all points in the spatial domain are feasible. In order to get a more realistic experiment, we consider next just a small subset of locations as feasible placements. For this experiment, we set $m=10$ yielding $m^2=100 $ total points inside the domain, of which, we consider eight specific points. Taking into account that the given points do not have to correspond to the mesh nodes, we took the closer mesh point to each of them. The considered points are:
    \begin{equation*}
    \begin{array}{lllllllllll}
    x_1&=(0.2,0.2)& \hspace{0.5cm}&x_2&=(0.5,0.4)& \hspace{0.5cm}&x_3&=(0.7,0.3)&\hspace{0.5cm}&x_4&=(0.8,0.0)\\
    x_5&=(0.8,1.0)& \hspace{0.5cm}&x_6&=(0.8,0.6)&\hspace{0.5cm}&x_7&=(0.4,0.9) &\hspace{0.5cm}&x_8&=(0.3,0.8).
    \end{array}
  \end{equation*}
Similarly to the first experiment, we are interested in observing how the structure of the placement vector $w$ and the time subintervals vector $\sigma$ change for different values of the penalization parameters. For this experiment we also set the sparsity enforcing penalty function $\Phi_{\epsilon}$ with $\epsilon=\frac{1}{8}$.

Table \ref{tab:304} shows the changes the structure of $w$ when we fix $\beta_\sigma=0$ and $\beta=0$, and let $\beta_w$ be the one that varies. Due to the small number of feasible points, the decrease of $\|w\|_{\ell_1}$ is not so agressive compared with the previous experiment. We can verify this behavior in Table \ref{tab:304-1}. Figure \ref{fig:310} shows graphically the location of these points, for different values of the penalization parameter.

      \begin{small}
      \begin{table}[ht]
      \centering
      \begin{tabular}{@{}cccccccc@{}}
      \toprule
      $\beta_w$ &$\#$ zeros in $w$ &$\mathbf{I_2}w$&$\mathbf{I_3}w$  & $\#$ ones in $w$ &$\#$ zeros in $\sigma$ & $\#$ ones in $\sigma$\\
      \midrule
$1\times10^{-5}$&0&0&0&8&0&14\\
$3\times10^{-5}$&0&0&5&3&0&14\\
$4\times10^{-5}$&0&3&2&3&0&14\\
$5\times10^{-5}$&0&5&0&3&0&14\\
$6\times10^{-5}$&0&5&0&3&0&14\\
$8\times10^{-5}$&0&5&1&2&0&14\\
$9.2\times10^{-5}$&5&1&0&2&0&14\\
$1.2\times10^{-4}$&7&0&0&1&0&14\\
$2\times10^{-4}$&8&0&0&0&14&0\\
      \bottomrule
      \end{tabular}
      \caption{Experiment 2. Changes in $w$'s structure with different values of $\beta_w$}\label{tab:304}
      \end{table}
      \end{small}

      \begin{small}
      \begin{table}[ht]
      \centering
      \begin{tabular}{@{}ccccccccc@{}}
      \toprule
      $\beta_w$ &$\mathbf{J_0}$ &
      $\mathbf{J_{\text{end}}}$ &$\mathbf{\| w \|}$ &$\mathbf{\| \sigma\|}$ &\textbf{iter} &\textbf{iter DA}&\textbf{PDE-solved}\\ \midrule
$1\times10^{-5}$&0.0033927&0.0033918&8&14&6&19&12\\
$3\times10^{-5}$&0.0035527&0.0035383&7.3243&14&6&19&12\\
$4\times10^{-5}$&0.0036327&0.0036056&6.9449&14&6&19&12\\
$5\times10^{-5}$&0.0037127&0.0036706&6.5887&14&6&19&12\\
$6\times10^{-5}$&0.0037927&0.0037351&6.3150&14&6&19&12\\
$8\times10^{-5}$&0.0039527&0.0038636&5.1545&14&6&19&12\\
$9.2\times10^{-5}$&0.0040327&0.0035329&2.3750&14&9&17&18\\
$1.2\times10^{-4}$&0.0042727&0.0034376&1&14&9&17&18\\
$2\times10^{-4}$&0.0049127&0.0033201&0&0&9&16&18\\
      \bottomrule
      \end{tabular}
      \caption{Experiment 2. Decreasing of $\|w\|$ for different values of $\beta_w$}\label{tab:304-1}
      \end{table}
      \end{small}

      \begin{figure}[ht]
      \begin{tikzpicture}[baseline, trim left=0.2cm, every node/.style={scale=0.7}]
      \begin{axis}[xmin=-0.1,xmax=1.1,ymin=-0.1,ymax=1.1,
      title=$(i)$,
      legend style={at={(0.5,-0.15)},anchor=north,legend columns=-1},width=5cm,height=5.2cm]
      \addplot[
      scatter, only marks, point meta=explicit symbolic, scatter/classes={
      1={mark=*,red,style={mark size=2pt}},%
      3={mark=*, blue,style={mark size=1.2pt}},
      2={mark=*,teal,style={mark size=0.7pt}},
      0={mark=*, violet,style={mark size=0.3pt}}},
      ]
      table[meta=label] {exp_pm_w1.dat};
      \end{axis}
      \end{tikzpicture}
      \begin{tikzpicture}[baseline, every node/.style={scale=0.7}]
      \begin{axis}[xmin=-0.1,xmax=1.1,ymin=-0.1,ymax=1.1,
      title=$(ii)$,
      legend style={at={(0.5,-0.15)},anchor=north,legend columns=-1},width=5cm,height=5.2cm]
      \addplot[
      scatter, only marks, point meta=explicit symbolic, scatter/classes={
      1={mark=*,red,style={mark size=2pt}},%
      3={mark=*, blue,style={mark size=1.2pt}},
      2={mark=*,teal,style={mark size=0.7pt}},
      0={mark=*, violet,style={mark size=0.3pt}}},
      ]
      table[meta=label] {exp_pm_w2.dat};
      \end{axis}
      \end{tikzpicture}
       \begin{tikzpicture}[baseline, every node/.style={scale=0.7}]
       \begin{axis}[xmin=-0.1,xmax=1.1,ymin=-0.1,ymax=1.1,
       title=$(iii)$,
       legend style={at={(0.5,-0.15)},anchor=north,legend columns=-1},width=5cm,height=5.2cm]
       \addplot[
       scatter, only marks, point meta=explicit symbolic, scatter/classes={
       1={mark=*,red,style={mark size=2pt}},%
       3={mark=*, blue,style={mark size=1.2pt}},
       2={mark=*,teal,style={mark size=0.7pt}},
       0={mark=*, violet,style={mark size=0.3pt}}},
       ]
       table[meta=label] {exp_pm_w3.dat};
       \end{axis}
       \end{tikzpicture}
      \begin{tikzpicture}[baseline, every node/.style={scale=0.7}]
      \begin{axis}[xmin=-0.1,xmax=1.1,ymin=-0.1,ymax=1.1,
      title=$(vi)$,
      legend style={at={(0.5,-0.15)},anchor=north,legend columns=-1},width=5cm,height=5.2cm]
      \addplot[
      scatter, only marks, point meta=explicit symbolic, scatter/classes={
      1={mark=*,red,style={mark size=2pt}},%
      3={mark=*, blue,style={mark size=1.2pt}},
      2={mark=*,teal,style={mark size=0.7pt}},
      0={mark=*, violet,style={mark size=0.3pt}}},
      ]
      table[meta=label] {exp_pm_w4.dat};
      \legend {{$1's$},{$I_3$},{$I_2$}},
      \end{axis}
      \end{tikzpicture}
       \begin{tikzpicture}[baseline, every node/.style={scale=0.7}]
       \begin{axis}[xmin=-0.1,xmax=1.1,ymin=-0.1,ymax=1.1,
       title=$(v)$,
       legend style={at={(0.5,-0.15)},anchor=north,legend columns=-1},width=5cm,height=5.2cm]
       \addplot[
       scatter, only marks, point meta=explicit symbolic, scatter/classes={
       1={mark=*,red,style={mark size=2pt}},%
       3={mark=*, blue,style={mark size=1.2pt}},
       2={mark=*,teal,style={mark size=0.7pt}},
       0={mark=*, violet,style={mark size=0.3pt}}},
       ]
       table[meta=label] {exp_pm_w5.dat};
       \legend {{$1's$},{$I_3$},{$I_2$}},
       \end{axis}
       \end{tikzpicture}
      \caption{Optimal placement vector's structure. Given points on the mesh. $(i) \beta_w=0.0001$, $(ii) \beta_w=0.0005$, $(iii) \beta_w=0.0008$, $(iv) \beta_w=0.00092$, $(v) \beta_w=0.00012$.}\label{fig:310}
      \end{figure}
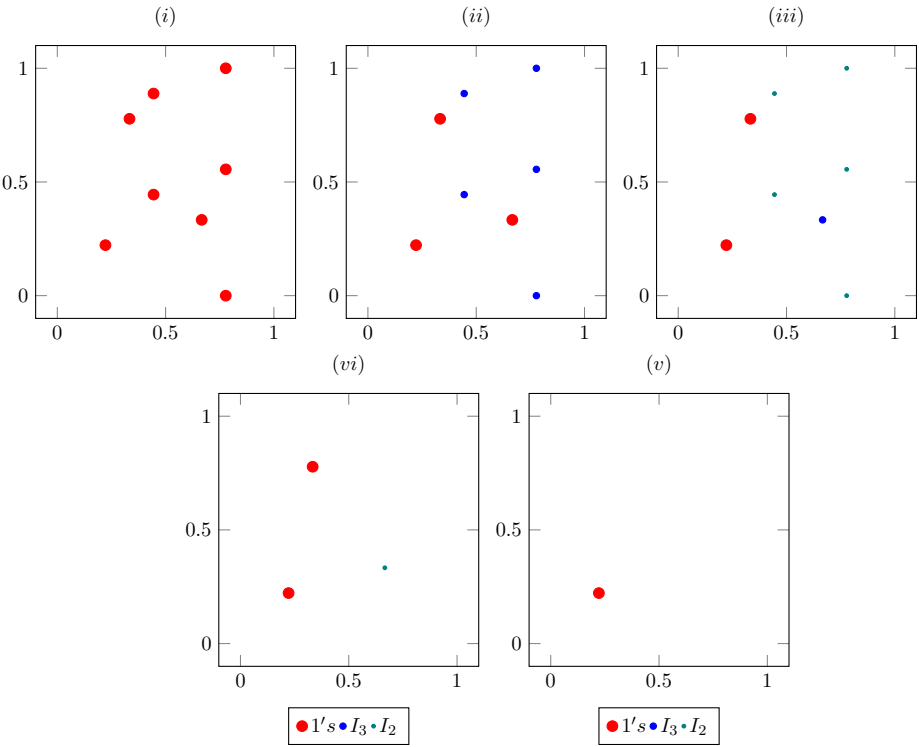

\subsection{Multiple training pairs}
The second set of experiments considers a training data set constituted by pairs, $\left(u_j^\dag, y_j^\dag\right), ~j=1, \dots, 9$. Given the initial condition $u^\dag$,  the state $y^\dag$ is obtained through simulation of the system model. Therefore, to form the training set, we just have to fix $u_j^\dag$. We build a training data set aiming for the reconstruction of a specific initial condition, (see e.g. \cite[Section 5.2]{ruthotto2018optimal}). We take the initial condition as the one given in the first set of experiments, i.e., $u^\dag(x,y)=\sin(2\pi x)\sin(2\pi y)$. Our training set consists in $N=9$ different functions that preserve some features of $u^\dag(x,y)$, and at the same time, add some modifications, e.g., translations of the maximum and minimum values, see Figure \ref{fig:training_set}. In practice, we do not have the exact initial condition that we want to rebuild. However, the background information available can be used to form a training data set.

\begin{figure}[ht]
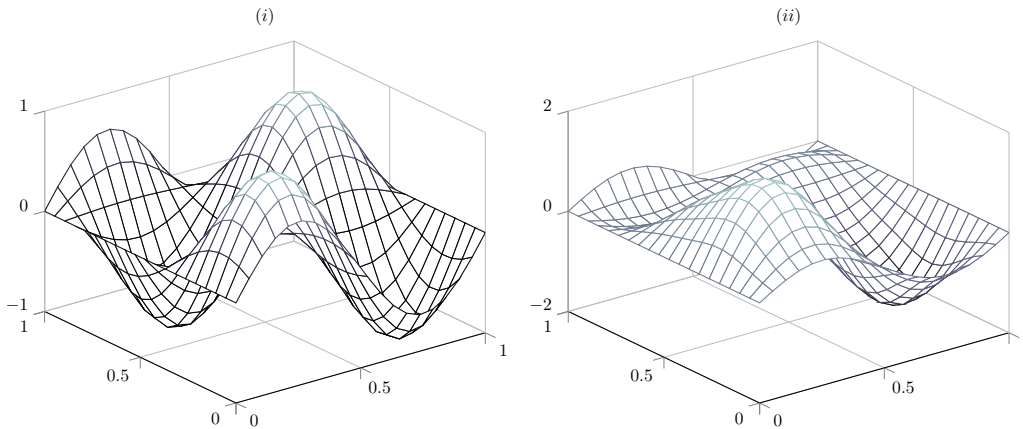

\centering
\input{g3_m20.tex}
\input{g8_m20.tex}
\caption{Training initial conditions. $(i)~ u_3^\dag(x,y)$,  $(ii)~ u_8^\dag(x,y)$}\label{fig:training_set}
\end{figure}

In meteorology, the initial condition is affected by the topography of the region of interest. Numerically we can model these perturbations as discontinuities. This modification consists in adding a jump function to each one of the $N$ elements in the training data set.

\subsubsection{Experiment 3}
For this experiment, we vary the parameter $\beta_w$ of the placement vector, while $\beta_\sigma=0.001$ and $\beta=0.001$ remain fix. We consider a sparsity enforcing penalty function $\Phi_\epsilon$ with $\epsilon=\frac{1}{8}$. In this case, the location vector $\mathbf{w}=(w,\sigma)$ has to be in average the optimal one for each training pair considered. Table \ref{tab:305} shows the structure of the resulting location vectors $w$ and $\sigma$. As in the first experiment, the value of $\|w\|_{\ell_1}$ decreases as the penalization parameter increases. We can observe this behavior, as well as the number of iterations in Table \ref{tab:305-1}.

      \begin{small}
      \begin{table}[ht]
      \centering
      \begin{tabular}{@{}cccccc@{}}
      \toprule
      $\beta_w$ &$\#$ zeros in $w$ &$\mathbf{I_2}w$ &$\#$ ones in $w$ &$\#$ zeros in $\sigma$ &$\#$ ones in $\sigma$\\
      \midrule
      0.0001&0&0&400&0&14\\
      0.001&4&0&396&0&14\\
      0.002&16&1&383&0&14\\
      0.004&77&0&323&0&14\\
      0.006&140&3&257&0&14\\
      0.008&173&0&227&0&14\\
      0.009&189&0&211&0&14\\
      0.010&204&0&196&0&14\\
      0.012&225&0&175&0&14\\
      0.014&244&0&156&0&14\\
      0.015&281&0&119&0&14\\
      0.016&321&0&79&0&14\\
      0.017&347&0&53&0&14\\
      0.018&359&0&41&0&14\\
      0.02&400&0&0&14&0\\
      \bottomrule
      \end{tabular}
      \caption{Experiment 3. Changes in $\mathbf{w}$'s structure for different values of $\beta_w$}\label{tab:305}
      \end{table}
      \end{small}

      \begin{small}
      \begin{table}[ht]
      \centering
      \begin{tabular}{@{}ccccccccc@{}}
      \toprule
      $\beta_w$ &$\mathbf{J_0}$ &
      $\mathbf{J_{\text{end}}}$&$\mathbf{\| w \|}$ &$\mathbf{\| \sigma\|}$ &\textbf{iter} &\textbf{iter DA}&\textbf{PDE-solved}\\ \midrule
      0.0001&0.17296&0.17296&400&14&6&21&108\\
      0.001&0.53296&0.52897&396&14&9&21&162\\
      0.002&0.93296&0.90068&383.487&14&13&21&234\\
      0.004&1.733&1.4253&323&14&9&21&162\\
      0.006&2.533&1.6865&257.899&14&13&21&234\\
      0.008&3.333&1.9506&227&14&8&21&144\\
      0.009&3.733&2.034&211&14&8&21&144\\
      0.010&4.133&2.0953&196&14&8&21&144\\
      0.012&4.934&2.2369&175&14&5&20&90\\
      0.014&5.734&2.3215&156&14&5&20&90\\
      0.015&6.134&2.4775&119&14&6&20&108\\
      0.016&6.534&2.0606&79&14&8&20&144\\
      0.017&6.934&1.4971&53&14&8&20&144\\
      0.018&7.334&1.0974&41&14&9&20&162\\
      0.02&8.133&0.1325&0&0&10&17&180\\
      \bottomrule
      \end{tabular}
      \caption{Experiment 3. Decreasing of $\|w\|$ for different values of $\beta_w$}\label{tab:305-1}
      \end{table}
      \end{small}

There are some cases where the obtained location vector has nonbinary weights. This result is not unexpected since the minimization problem, that includes the sparsity enforcing penalty function, is still an approximation of the mixed integer nonlinear problem. Since these nonbinary weights are very few (see Table \ref{tab:305-1}, $\beta_w=0.002$ and $\beta_w=0.006$), we decide if they correspond to 0 or 1 by using an exhaustive search. We show this results for these values of $\beta_w$ in Table \ref{tab:305_binary}. The graphical representation of the resulting location vectors with different values of $\beta_w$ is presented in Figure \ref{fig:311}. Alternatively, it is possible to use  thresholding techniques to ensure binary weights.

\begin{small}
\begin{table}[ht]
\centering
\begin{tabular}{@{}cccccc@{}}
\toprule
$\beta_w$ &N$^{\circ}$ null elements& $\mathbf{I_2}w$&N$^{\circ}$ elements&N$^{\circ}$ null elements& N$^{\circ}$ elements\\
&in $w$ & & one in $w$ & in $\sigma$ & one in $\sigma$\\
\midrule
0.002&16&0&384&0&14\\
0.006&143&0&257&0&14\\
\bottomrule
\end{tabular}
\caption{Experiment 3. Recovering the binary structure of $w$}\label{tab:305_binary}
\end{table}
\end{small}

      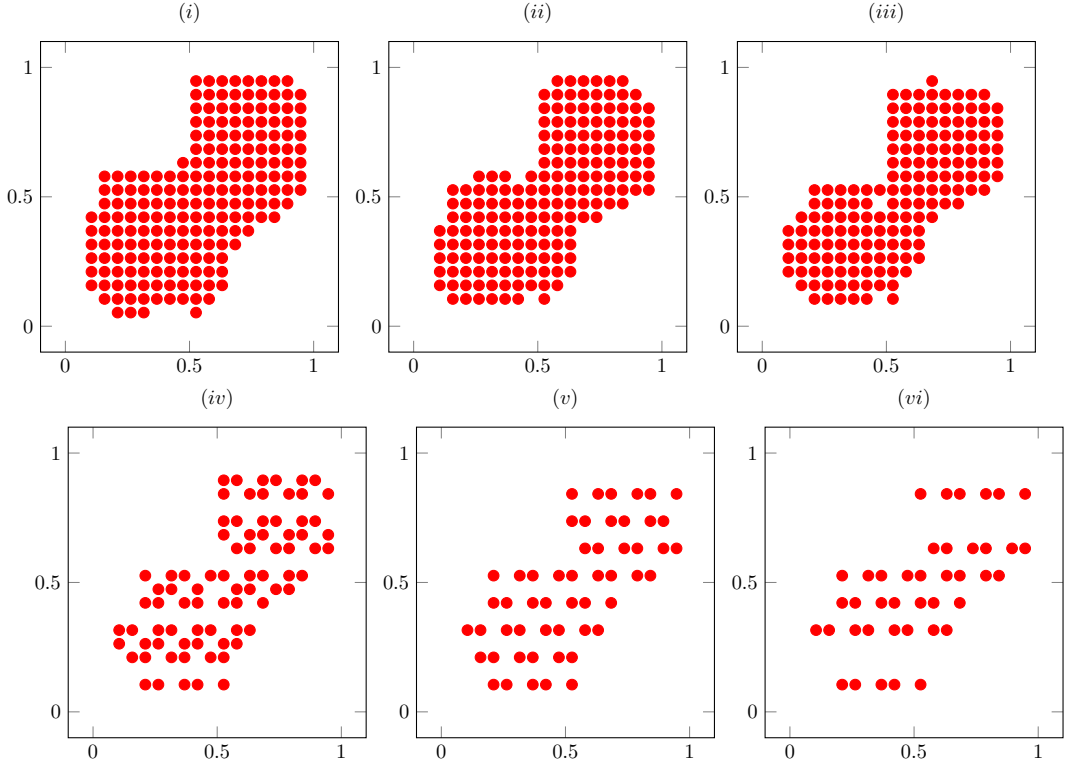
\begin{figure}[ht]
      \begin{tikzpicture}[baseline, trim left=0.2cm, every node/.style={scale=0.7}]
      \begin{axis}[xmin=-0.1,xmax=1.1,ymin=-0.1,ymax=1.1,
      title=$(i)$,
      legend style={at={(0.5,-0.15)},anchor=north,legend columns=-1},width=5.5cm,height=5.7cm]
      \addplot[
      scatter, only marks, point meta=explicit symbolic, scatter/classes={
      1={mark=*,red,style={mark size=2pt}},%
      3={mark=*, blue,style={mark size=1pt}},%
      2={mark=*,teal,style={mark size=0.5pt}},%
      0={mark=*, violet,style={mark size=0.1pt}}},%
      ]
      table[meta=label] {ts_w5_m20.dat};
      \end{axis}
      \end{tikzpicture}
      \begin{tikzpicture}[baseline, every node/.style={scale=0.7}]
      \begin{axis}[xmin=-0.1,xmax=1.1,ymin=-0.1,ymax=1.1,
      title=$(ii)$,
      legend style={at={(0.5,-0.15)},anchor=north,legend columns=-1},width=5.5cm,height=5.7cm]
      \addplot[
      scatter, only marks, point meta=explicit symbolic, scatter/classes={
      1={mark=*,red,style={mark size=2pt}},%
      3={mark=*, blue,style={mark size=1pt}},%
      2={mark=*,teal,style={mark size=0.5pt}},%
      0={mark=*, violet,style={mark size=0.1pt}}},%
      ]
      table[meta=label] {ts_w6_m20.dat};
      \end{axis}
      \end{tikzpicture}
      \begin{tikzpicture}[baseline, every node/.style={scale=0.7}]
      \begin{axis}[xmin=-0.1,xmax=1.1,ymin=-0.1,ymax=1.1,
      title=$(iii)$,
      legend style={at={(0.5,-0.15)},anchor=north,legend columns=-1},width=5.5cm,height=5.7cm]
      \addplot[
      scatter, only marks, point meta=explicit symbolic, scatter/classes={
      1={mark=*,red,style={mark size=2pt}},%
      3={mark=*, blue,style={mark size=1pt}},%
      2={mark=*,teal,style={mark size=0.5pt}},%
      0={mark=*, violet,style={mark size=0.1pt}}},%
      ]
      table[meta=label] {ts_w7_m20.dat};
      \end{axis}
      \end{tikzpicture}
      \begin{tikzpicture}[baseline, trim left=0.2cm, every node/.style={scale=0.7}]
      \begin{axis}[xmin=-0.1,xmax=1.1,ymin=-0.1,ymax=1.1,
      title=$(iv)$,
      legend style={at={(0.5,-0.15)},anchor=north,legend columns=-1},width=5.5cm,height=5.7cm]
      \addplot[
      scatter, only marks, point meta=explicit symbolic, scatter/classes={
      1={mark=*,red,style={mark size=2pt}},%
      3={mark=*, blue,style={mark size=1pt}},%
      2={mark=*,teal,style={mark size=0.5pt}},%
      0={mark=*, violet,style={mark size=0.1pt}}},%
      ]
       table[meta=label] {ts_w9_m20.dat};
      \end{axis}
      \end{tikzpicture}
      \begin{tikzpicture}[baseline, every node/.style={scale=0.7}]
      \begin{axis}[xmin=-0.1,xmax=1.1,ymin=-0.1,ymax=1.1,
      title=$(v)$,
      legend style={at={(0.5,-0.15)},anchor=north,legend columns=-1},width=5.5cm,height=5.7cm]
      \addplot[
      scatter, only marks, point meta=explicit symbolic, scatter/classes={
      1={mark=*,red,style={mark size=2pt}},%
      3={mark=*, blue,style={mark size=1pt}},%
      2={mark=*,teal,style={mark size=0.5pt}},%
      0={mark=*, violet,style={mark size=0.1pt}}},%
      ]
      table[meta=label] {ts_w10_m20.dat};
      \end{axis}
      \end{tikzpicture}
      \begin{tikzpicture}[baseline, every node/.style={scale=0.7}]
      \begin{axis}[xmin=-0.1,xmax=1.1,ymin=-0.1,ymax=1.1,
      title=$(vi)$,
      legend style={at={(0.5,-0.15)},anchor=north,legend columns=-1},width=5.5cm,height=5.7cm]
      \addplot[
      scatter, only marks, point meta=explicit symbolic, scatter/classes={
      1={mark=*,red,style={mark size=2pt}},%
      3={mark=*, blue,style={mark size=1pt}},%
      2={mark=*,teal,style={mark size=0.5pt}},%
      0={mark=*, violet,style={mark size=0.1pt}}},%
      ]
      table[meta=label] {ts_w11_m20.dat};
      \end{axis}
      \end{tikzpicture}
      \caption{Optimal placement vector's structure. Different values of $\beta_w$. $(i)\beta_w=0.010$, $(ii) \beta_w=0.012$, $(iii) \beta_w=0.014$, $(iv) \beta_w=0.016$, $(v)\beta_w=0.017$, $(vi) \beta_w=0.018$.}\label{fig:311}
      \end{figure}

  \subsection{Error in the reconstruction of the desired state}
The aim of this last experiment is computing the error between the desired state and the one obtained by using the rebuilt initial condition. To do that we form the training set as follows: We consider the same initial condition for every element in the training set, i.e., $u^\dag_j(x,y)=u^\dag(x,y)$ for all $j=1,\ldots, N$. As in the previous experiment, $y^\dag_j$ is obtained through simulations of the system model, however, in this case, we consider different right-hand sides in the dynamical system of the data assimilation problems. Moreover, for each $j=1,\ldots, N$, the observed state $z_{oj}$ is built by adding Gaussian noise with mean zero and different values of standard deviation. Doing this let us simulate the covariance matrix present in the model.

  \subsubsection{Experiment 4}
For this experiment, every element in the traing set has the form  $u^\dag(x,y)=\sin(2\pi x)\sin(2\pi y)$. We work with different values of standard deviation (SD) in the observed state while keeping fix the values of the penalization parameters $\beta_\sigma=1\times 10^{-6}$ and $\beta=0$. The absolute and relative errors between the desired and the obtained states are compute in the following way:
\[
\text{error.abs}=\frac{1}{N}\sum_{j=1}^N\|y^\dag_j - y_j\|_{L^2(Q)},\qquad \text{error.rel}=\frac{\sum_{j=1}^N\|y^\dag_j - y_j\|_{L^2(Q)}}{\sum_{j=1}^N\| y^\dag_j\|_{L^2(Q)}}
\]
We report the absolute and the relative errors obtained with different noise levels in Table \ref{tab:405}.

  \begin{small}
  \begin{table}[ht]
  \centering
  \begin{tabular}{@{}ccccc@{}}
  \toprule
  $\beta_w$ &SD &
  error ab.&error rel. \\ \midrule
   0.01&\multirow{3}{*}{0.001}&$3.776\times10^{-4}$&0.03319\\
  0.05&&$3.746\times10^{-4}$&0.03293\\
  0.1&&$3.797\times10^{-4}$&0.03338\\\midrule
  0.01&\multirow{3}{*}{0.01}&$1.868\times10^{-3}$&0.16437\\
  0.05&&$1.866\times10^{-3}$&0.16404\\
  0.1&&$1.875\times10^{-3}$&0.16495\\
  \bottomrule
  \end{tabular}
  \caption{Experiment 4. Error between the desired and the obtained states}\label{tab:405}
  \end{table}
  \end{small}

As expected, using larger values of standard deviation in the noise added to the observed state $z_{oj}$, for all $j=1,\ldots,N$, leads to larger values of the absolute and relative errors.

\section{Conclusions}
In this paper, we propose a bilevel learning approach for observation placement in variational data assimilation. The solution of our bilevel optimization problem consists of two location vectors both in space and time. The first one provides the optimal configuration where the sensors/devices have to be placed, while the second one gives the optimal time subintervals at which the sensors have to be turned on.

We divide the study of the problem into two stages. The first one focuses on the analysis of a semilinear data assimilation problem, whose adjoint equation contains regular measures in space and mollified Dirac measures in time. We show the existence of a very weak solution of the adjoint state as well as useful estimates in terms of the solution vectors.

Thereafter, we use the optimality condition of the data assimilation problem as constraint of the optimal location problem and consider a supervised learning approach in which we presuppose the existence of a training set of initial conditions and their corresponding states. A first-order optimality system for the bilevel problem is then derived by considering the constraint as a multi-state system of partial differential equations, for which an adapted penalty approach is investigated.

Numerically, both the lower- and the upper-level problems are solved by using second-order methods. Specifically, the upper-level problem is solved by using a projected BFGS method where the approximation of the inverse of the reduced Hessian matrix is built iteratively by using the estimation of $\epsilon-$active sets. We also show that the proposed projected algorithm preserves the superlinear convergence rate.

The performed experiments provide the structure of the locations vectors, showing that the solution becomes sparser as the penalization parameters increases. In some applications, one may wish to take into account economic factors and non-viable places to locate sensors or devices. We consider these scenarios in some experiments to get more realistic results.



\bibliographystyle{plain}
\bibliography{referencias}

\end{document}